\documentclass[12pt]{amsart}
\usepackage{}
\usepackage{amsmath}
\usepackage{amsfonts}
\usepackage{amssymb}
\usepackage{ifpdf}
\ifpdf
  \usepackage[colorlinks,final,backref=page,hyperindex]{hyperref}
\else
  \usepackage[colorlinks,final,backref=page,hyperindex,hypertex]{hyperref}
\fi
\usepackage[active]{srcltx}

\begin{document}
\newcommand {\emptycomment}[1]{} 

\baselineskip=14pt
\newcommand{\nc}{\newcommand}
\newcommand{\delete}[1]{}
\nc{\mfootnote}[1]{\footnote{#1}} 
\nc{\todo}[1]{\tred{To do:} #1}

\nc{\mlabel}[1]{\label{#1}}  
\nc{\mcite}[1]{\cite{#1}}  
\nc{\mref}[1]{\ref{#1}}  
\nc{\mbibitem}[1]{\bibitem{#1}} 

\delete{
\nc{\mlabel}[1]{\label{#1}  
{\hfill \hspace{1cm}{\bf{{\ }\hfill(#1)}}}}
\nc{\mcite}[1]{\cite{#1}{{\bf{{\ }(#1)}}}}  
\nc{\mref}[1]{\ref{#1}{{\bf{{\ }(#1)}}}}  
\nc{\mbibitem}[1]{\bibitem[\bf #1]{#1}} 
}

\newtheorem{thm}{Theorem}[section]
\newtheorem{pro}[thm]{Proposition}
\newtheorem{lem}[thm]{Lemma}
\newtheorem{cor}[thm]{Corollary}
\newtheorem{claim}{Claim}[section]
\theoremstyle{definition}
\newtheorem{defi}[thm]{Definition}
\newtheorem{rem}[thm]{Remark}
\newtheorem{exa}[thm]{Example}

\renewcommand{\labelenumi}{{\rm(\alph{enumi})}}
\renewcommand{\theenumi}{\alph{enumi}}

\nc{\tred}[1]{\textcolor{red}{#1}}
\nc{\tblue}[1]{\textcolor{blue}{#1}}
\nc{\tgreen}[1]{\textcolor{green}{#1}}
\nc{\tpurple}[1]{\textcolor{purple}{#1}}
\nc{\btred}[1]{\textcolor{red}{\bf #1}}
\nc{\btblue}[1]{\textcolor{blue}{\bf #1}}
\nc{\btgreen}[1]{\textcolor{green}{\bf #1}}
\nc{\btpurple}[1]{\textcolor{purple}{\bf #1}}

\newcommand{\gl}{\mathfrak {gl}}

\def \i {\mathrm{i}}
\def \pd {\phi_{_d}}
\def \hls {\ensuremath{(\v,\cdot,\psi)}}
\def \o {\ensuremath{\mathcal{O}}}
\def \rl {\ensuremath{r^{\sharp}}}
\def \srl {\ensuremath{r_{21}^{\sharp}}}
\def \e {\eta}
\def \pr {^{\prime}}
\def \a{\alpha}
\def \b{\beta}
\def \r {\rho}
\def \rs {\rho^*}
\def \ro {\rho^\circ}
\nc{\rmi}{{\mathrm{I}}} \nc{\rmii}{{\mathrm{II}}}
\nc{\rmiii}{{\mathrm{III}}}
\def\t{\otimes}
\def\d{\Delta}
\def\ds {\Delta^*}
\def\la{\langle}
\def\ra{\rangle}
\renewcommand {\v} {\ensuremath{V}}
\def \bb {\mathfrak{B}}
\def\k{\mathfrak{k}}
\def\h{\mathfrak{h}}
\def\g{\mathfrak{g}}
\def\gs{\mathfrak{g}^*}
\def\ad{\mathrm{ad}}
\def\ads{\mathrm{ad}^*}
\def\ado{\mathrm{ad}^{\circ}}
\nc \adx{\ad_x} \nc \adsx{\ads_x} \nc \adox{\ado_x}
\def\da{\mathfrak{ad}}
\def\das{\mathfrak{ad}^*}
\def\dao{\mathfrak{ad}^{\circ}}
\def\xg{_\g}
\def\p{\phi_{_\g}}
\def\pp{\phi_{_{\g\pr}}}
\def\ps {\phi_{_\g}^*}
\def\j{\wedge}
\def\lr{\longrightarrow}
\def \id{\mathrm{Id}}
\def \bra {[\cdot,\cdot]\xg}
\def \hlg {\ensuremath{(\g,[\cdot,\cdot]\xg,\p)}}
\def \hlgs {\ensuremath{(\gs,[\cdot,\cdot]_{\gs},\ps)}}
\def \hlk {(\k,[\cdot,\cdot]_\k,\phi_{_\k})}
\def \lj {\j\cdots\j}
\def \lt {\t\cdots\t}
\nc \x [1] {\ensuremath{x_{#1}}}
\nc \y [1]{\ensuremath{y_{#1}}}
\nc \xqh[1]  {\ensuremath{\sum\limits_{#1}}}
\nc \sxqh[2] {\ensuremath{\sum\limits_{#1}^{#2}}}
\newcommand \rep {\ensuremath{(\v,\b,\r)}}
\nc \drep {\ensuremath{(\v^*,\b^*,\ro)}}
\nc \tc {type-C}
\nc{\jac}{\mathrm{Jac}}

\font\cyr=wncyr10

\nc{\redtext}[1]{\textcolor{red}{#1}}
\nc{\li}[1]{\textcolor{blue}{#1}}
\nc{\lir}[1]{\textcolor{blue}{Li: #1}}
\nc{\cm}[1]{\textcolor{red}{Chengming: #1}}
\nc{\ty}[1]{\textcolor{blue}{Tao: #1}}


\title{Another approach to Hom-Lie bialgebras via Manin triples}

\author{Yi Tao}
\address{Chern Institute of Mathematics\& LPMC, Nankai University, Tianjin 300071, China}
    \email{taoyinku@163.com}

\author{Chengming Bai}
\address{Chern Institute of Mathematics \& LPMC, Nankai University, Tianjin 300071, China}
         \email{baicm@nankai.edu.cn}

\author{Li Guo}
\address{Department of Mathematics and Computer Science,
         Rutgers University,
         Newark, NJ 07102}
\email{liguo@rutgers.edu}

\date{\today}

\begin{abstract}                                                                        In this paper, we study Hom-Lie bialgebras
by a new notion of the dual representation of a representation of
a Hom-Lie algebra. Motivated by the essential connection between
Lie bialgebras and Manin triples, we introduce the notion of a Hom-Lie bialgebra with emphasis on its compatibility with
a Manin triple of Hom-Lie algebras associated to a
nondegenerate symmetric bilinear form satisfying a new invariance
condition. With this notion, coboundary Hom-Lie bialgebras can be
studied without a skew-symmetric condition of $r\in\g\otimes \g$,
naturally leading to the classical Hom-Yang-Baxter equation whose
solutions are used to construct coboundary Hom-Lie bialgebras. In
particular, they are used to obtain a canonical Hom-Lie bialgebra
structure on the double space of a Hom-Lie bialgebra. We also
derive solutions of the classical Hom-Yang-Baxter equation from
\o-operators and Hom-left-symmetric algebras.
\end{abstract}


\subjclass[2010]{16T10, 16T25, 17A30, 17B62}

\keywords{Hom-Lie algebra, bialgebra, matched pair, Manin triple, classical Hom-Yang-Baxter
equation, $\mathcal O$-operator}

\maketitle

\tableofcontents

\numberwithin{equation}{section}

\allowdisplaybreaks

\section{Introduction}
The notion of Hom-Lie algebras was introduced in \cite{HLS} in the
context of deformations of the Witt algebra and the Virasoro
algebra. In a Hom-Lie algebra $\g$, the Jacobi identity defining a
Lie algebra is twisted by a linear map $\p$, called the Hom-Jacobi
identity, which recovers the Jacobi identity in the special case
when $\p$ is the identity map. With this generalization of the Lie
algebra, some $q$-deformations of the Witt and the Virasoro
algebras have the structure of a Hom-Lie algebra \cite{HLS}. Due
to their close relationship with discrete and deformed vector
fields and differential calculus~\cite{HLS,LS1,LS2}, Hom-Lie
algebras have been studied in broad areas
\cite{AMM,BM,BEM,CWZ,CS-purely,MDL,MLY,MS1,S,BS,SD,Y1,DY-hlbi}.

While the generalization of Hom-Lie algebra from Lie algebra allows wider potential applications, this generalization has often posed challenges in extending the well-established theory of Lie algebras.

This is the case of the bialgebra theory which is an essential aspect of Lie algebra due to its important applications. For example, the Lie bialgebra is the algebraic structure corresponding to a Poisson-Lie group and the classical structure of a quantized universal enveloping algebra~\mcite{CP,D}.
A Lie bialgebra consists of a Lie algebra $(\g,[\cdot,\cdot]_\g$)
where $[\cdot,\cdot]_\g:\otimes^2\g\to \g$ is a Lie bracket, a Lie
coalgebra $(\g,\Delta)$ where $\Delta:\g\to \otimes^2\g$ is a Lie
comultiplication, and a suitable compatibility condition between
the Lie bracket $[\cdot,\cdot]_\g$ and the Lie comultiplication
$\Delta$.

The great importance of the Lie bialgebra in the related fields
relies on the fact that a Lie bialgebra can be equivalently
described in terms of a Manin triple. More precisely, the
compatibility condition of a Lie bialgebra is given by the
condition that the Lie algebra $\g$ and the Lie algebra $\g^*$
from the Lie coalgebra $(\g,\Delta)$ are
subalgebras of a third Lie algebra $\g\oplus \g^*$ such that, for
the usual pairing $\langle\cdot,\cdot \rangle$ between $\g$ and
$\g^*$, the bilinear form
\begin{equation}\label{eq:liebidb}
\bb( x+u^*, y+v^*):=\langle x, v^*\rangle+\langle u^*,y\rangle,\;\;\forall x,y\in \g, u^*,v^*\in \g^*,
\end{equation}
is invariant on $\g\oplus \g^*$ in the sense that
\begin{equation}
\bb([a,b]_{\g\oplus\g^*}, c)=\bb(a,[b,c]_{\g\oplus\g^*}),\;\;\forall a,b,c\in \g\oplus\g^*.
\label{eq:inv}
\end{equation}

Thus a suitable generalization of the theory of Lie bialgebras to the Hom context should include a meaningful generalization of Manin triples in the picture, which turns out to be a nontrivial requirement.
A notion of Hom-Lie bialgebra was proposed in \cite{DY-hlbi}, where the compatibility condition is that
the Hom-Lie cobracket $\d$ is a 1-cocycle on the Hom-Lie algebras
$\g$ with the coefficients in $\j^2\g$:
\begin{equation}\label{CC1}\d[x,y]\xg=\ad_{\p(x)}\d(y)-\ad_{\p(y)}\d(x),\;\;\forall x,y\in \g.\end{equation}
Due to the lack of a Manin triple interpretation of this
definition, other notions of Hom-Lie bialgebras were subsequently
introduced~\cite{CS-purely,BS}, incorporating Hom variations of the
nondegenerate symmetric bilinear form in Eq.~\eqref{eq:liebidb}.

Different ways of integrating the Hom-map $\p$ into the bilinear form in Eq.~\eqref{eq:liebidb}
give rise to variations in generalizing the invariant nondegenerate
symmetric bilinear forms and the resulting Manin triples to the
Hom case, yielding various notions of Hom-Lie bialgebras.
Thus a Hom-Lie bialgebra is a triple $(\g,[\cdot,\cdot]_\g, \Delta)$ such that $(\g,[\cdot,\cdot]_\g)$ defines a Hom-Lie algebra and $(\g, \Delta)$ defines a Hom-Lie coalgebra (that is, $\Delta^*$ defines a Hom-Lie algebra on $\g^*$ when $\g$ is finite-dimensional) which satisfy certain compatibility conditions depending on the appearance of $\p$.
We compare the approaches in~\cite{BS} and~\cite{CS-purely} in order to motivate the approach taken in this paper.

In \cite{BS}, the authors introduced a notion of Hom-Lie
bialgebras on Hom-Lie algebras which are {\bf admissible} in the sense that the following equations hold:
\begin{eqnarray}
  ~\label{eqn:coadjointrepcon1}[({\rm Id}-\p^2) (x),\p(y)]_\g&=&0,\\
~\label{eqn:coadjointrepcon2}[({\rm Id}-\p^2)(x),[\p(y),z]_\g]_\g&=&[({\rm
Id}-\p^2)(y),[\p(x),z]_\g]_\g,\forall x,y,z\in \g.
\end{eqnarray}
The compatibility condition is still given by Eq.~(\ref{CC1}), but is only
required to hold on a subspace of $\j^2\gs$, that is,
$$\la \d[x,y]\xg,\ps(\xi)\j \e\ra=\la \ad_{\p(x)}\d(y)- \ad_{\p(y)}\d(x),\ps(\xi)\j \e\ra,\;\;\forall x,y\in \g, \xi,\eta\in \g^*.$$
It is verified in~\cite{BS} that this variation of Hom-Lie bialgebras correspond to the Manin triples of Hom-Lie algebras associated to the nondegenerate symmetric bilinear form satisfying
\begin{equation}\label{eq:bb1}
\bb([a,b]_{\g\oplus\g^*},c)=\bb(a,[b,c]_{\g\oplus\g^*}),\
\bb(\phi_{_{\g\oplus\g^*}}(a),b)=\bb(a,\phi_{_{\g\oplus\g^*}}(b)), \;\;\forall a,b,c\in \g\oplus \g^*.
\end{equation}

In \cite{CS-purely}, the authors introduced another notion of
Hom-Lie bialgebras, called purely Hom-Lie bialgebras, defined on Hom-Lie algebras which are regular, that is,
such that $\p$ is invertible. The compatibility condition is
changed to
$$\d[x,y]\xg=\ad_{\p^{-1}(x)}\d(y)-\ad_{\p^{-1}(y)}\d(x),\forall x,y\in \g.$$
Such Hom-Lie bialgebras correspond to the Manin triples of Hom-Lie
algebras associated to the nondegenerate symmetric bilinear form satisfying
\begin{equation}\label{eq:bb2}
\bb([a,b]_{\g\oplus\g^*},\phi_{_{\g\oplus\g^*}}(c))=\bb(\phi_{_{\g\oplus\g^*}}(a),[b,c]_{\g\oplus \g^*}),\
\bb(\phi_{_{\g\oplus\g^*}}(a),\phi_{_{\g\oplus\g^*}}(b))=\bb(a,b),
\end{equation}
for any $ a,b,c\in \g\oplus \g^*$.

As can be noted at this point, there should be a variety of
Hom-Lie bialgebras and it is worthwhile to explore other
possibilities in order to gain further insight on the interesting
role played by the Hom-map $\p$ in Hom-Lie bialgebras.  In this
paper, we consider a new kind of Hom-Lie bialgebras which
correspond to the Manin triples of Hom-Lie algebras associated to
the nondegenerate symmetric bilinear form satisfying the condition
\begin{equation}\label{invariant}
\bb([a,b]_{\g\oplus\g^*},c)=\bb(a,[\phi_{_{\g\oplus\g^*}}(b),c]_{\g\oplus\g^*}),\
\bb(\phi_{_{\g\oplus\g^*}}(a),b)=\bb(a,\phi_{_{\g\oplus\g^*}}(b)),
\end{equation}
for any $a,b,c\in \g\oplus\g^*$. Note that if $\phi_{_{\g\oplus\g^*}}$ is invertible, then Eq.~(\ref{invariant}) is equivalent to
\begin{equation}\bb([a,b]_{\g\oplus\g^*},\phi_{_{\g\oplus\g^*}}(c))=\bb(\phi_{_{\g\oplus\g^*}}(a),[b,c]_{\g\oplus\g^*}),\
\bb(\phi_{_{\g\oplus\g^*}}(a),b)=\bb(a,\phi_{_{\g\oplus\g^*}}(b)),\end{equation}
for any  $a,b,c\in \g\oplus\g^*$, that is, a certain hybrid of
Eqs.~(\ref{eq:bb1}) and (\ref{eq:bb2}).
It is also interesting to note that the
compatible condition is exactly Eq.~(\ref{CC1}), on weakly
involutive Hom-Lie algebras. Thus this new kind of Hom-Lie
bialgebras incorporate features of all the previous three notions
of Hom-Lie bialgebras.

Another complexity in defining Hom-Lie bialgebras is their
dependence on a suitable notion of the dual representation, since
the natural choice $(\v^*,\beta^*,\rs)$ for the dual
representation of a representation $(\v,\beta,\rho)$ of a Hom-Lie
algebra does not work for all Hom-Lie algebras~\cite{BM}. To work
around this difficulty, the authors of~\cite{BS} imposed a
condition on Hom-Lie bialgebras called admissible such that
$(\v^*,\beta^*,\rs)$ is still the dual representation of
$(\v,\beta,\rho)$. Taking a different approach, a notion of dual
representation, under the condition that $\beta$ is invertible,
was introduced in~\cite{CS-purely} in order to
define their Hom-Lie bialgebras. In this paper,  we introduce an
alternative notion of the dual representation  \drep, where
$\ro=\rho^*\,  \p\footnote{The composition sign $\circ $ is suppressed throughout the paper to simplify notations.} :\g\lr\gl(V^*)$, under the
weakly involutive condition $\r(\p^2(x))=\r(x)$. In particular, for the adjoint
representation $(\g,\p,\ad)$, $\ad(\p^2(x))=\ad(x)$ for all $x\in
\g$ gives the notion of a weakly involutive Hom-Lie algebra.

One importance of Lie bialgebras lies with the essential role played by coboundary Lie bialgebras and quasitriangular Lie bialgebras which led to
the introduction of the classical Yang-Baxter equation whose solutions produce such bialgebras.
Another advantage of our approach of Hom-Lie bialgebras is that it works, without any restrictions, with the coboundary Hom-Lie bialgebras and quasitriangular Hom-Lie bialgebras and their connection with the classical Hom-Yang-Baxter equation, while conditions like $r\in\g\otimes \g$ being skew-symmetric needed to be imposed in~\cite{CS-purely,BS}.
In particular, there is a canonical (quasitriangular) Hom-Lie bialgebra structure on the double space of a Hom-Lie bialgebra induced from a non-skew-symmetric solution of the classical Hom-Yang-Baxter equation.

The paper is organized as follows. In Section~\ref{sec:rep}, we
introduce a new dual representation  \drep\,, leading to the basic
notion of weakly involutive Hom-Lie algebra. This weakly
involutive condition is shown to provide the right context for a
nondegenerate symmetric invariant bilinear form on a Hom-Lie
algebra (Proposition~\ref{invariant and weakly involutive}). These
nondegenerate symmetric invariant bilinear forms are applied in
Section~\ref{sec:match} to study matched pairs and Manin triples
of weakly involutive Hom-Lie algebras. With these preparations, we
are naturally led to our notion of a Hom-Lie bialgebra together
with its characterizations by matched pairs and Manin triples
(Theorem~\ref{thm:equiv}). As a further evidence that we have a
suitable notion of Hom-Lie bialgebras, in Section~\ref{sec:quasi},
we study coboundary Hom-Lie bialgebras and quasitriangular Hom-Lie
bialgebras in terms of $r$-matrices (Theorem~\ref{main theorem}),
showing that skew-symmetric solutions of the classical
Hom-Yang-Baxter equation in a weakly involutive Hom-Lie algebra
naturally give rise to Hom-Lie bialgebras. Finally, in
Section~\ref{sec:oop}, we study the operator forms of
the classical Hom-Yang-Baxter equation and then construct
solutions of the classical Hom-Yang-Baxter equation in semidirect
products of weakly involutive Hom-Lie algebras using \o-operators
(Corollary \ref{cor, coboundary Hom-Lie bi, r}) and
Hom-left-symmetric algebras (Corollary \ref{solutions in hls}).
\smallskip

\noindent
\textbf{Notations.} Throughout this paper, $\g$ denotes a Hom-Lie algebra and $V$ denotes a vector space, both assumed to be finite-dimensional. $\gs$ and $V^*$ denote the dual of $\g$ and $V$ respectively. Further $x,y,z,x_i,y_i$ denote elements in $\g$, $a,b,c$ denote elements in $\gs$, $u,v,w$ denote elements in $V$ and $\xi,\e$ denote elements in $V^*$. $\ad$ and $\da$ are the adjoint representations of the Hom-Lie algebras $\g$ and $\gs$ respectively.

\section{Hom-Lie algebras, their representations and bilinear forms}
\label{sec:rep}

In this section, we first recall some basic concepts and facts on
Hom-Lie algebras and their representations~\cite{HLS,S}. We then
introduce a new notion of the dual representation of a representation of a
Hom-Lie algebra. The related bilinear forms are also studied.

\begin{defi}[\cite{HLS}]
    A {\bf Hom-Lie algebra} is a triple $(\g,[\cdot,\cdot]\xg,\p)$ consisting of a linear space $\g$, a skew-symmetric bilinear map (bracket) $\bra:\g\t\g\lr\g$ and a bracket preserving linear map $\p:\g\lr\g$ satisfying the {\bf Hom-Jacobi identity}
    \begin{equation}
    [\p(x),[y,z]\xg]\xg+[\p(y),[z,x]\xg]\xg+[\p(z),[x,y]\xg]\xg=0.
    \end{equation}
\end{defi}

A Hom-Lie algebra $\hlg$ is called {\bf involutive} if $\p$ satisfies $\p^2=\id$.

\begin{defi}
A homomorphism of Hom-Lie algebras $f:\hlg\lr\hlk$ is a linear map
$f:\g\lr\k$ such that
\begin{eqnarray}
f([x,y]\xg)&=& [f(x),f(y)]_\k, \forall x, y\in \g, \\
f\,  \p &=& \phi_{_\k}\,  f.
\end{eqnarray}
\end{defi}

The notion of a representation of a Hom-Lie algebra
$\hlg$ was introduced in~\cite{S}.
\begin{defi}
    A {\bf representation} of a Hom-Lie algebra $\hlg$ is a triple $(\v,\b,\r)$ consisting of a vector space $V$, an element $\b\in\gl(\v)$ and a linear map $\r:\g\lr\gl(\v)$, such that for any $x, y\in\g$, the following equalities hold:
    \begin{itemize}
        \item[(i)] $\r(\p(x))\, \b=\b\, \r(x)$;
        \item[(ii)] $\r([x,y]\xg)\, \b=\r(\p(x))\, \r(y)-\r(\p(y))\, \r(x)$.
    \end{itemize}
\end{defi}

It is straightforward to
see that $(\g,\p,\ad)$, where $\ad_x (y):=[x,y]$ for all $x,y\in \g$, is a representation of $\hlg$, called the {\bf adjoint
representation}.

As in the Lie algebra case, given a representation \rep\  of a Hom-Lie algebra $\hlg$, we obtain a new Hom-Lie algebra $\g\ltimes_\r
\v:=(\g\oplus\v,[\cdot,\cdot]_{\g\ltimes_\r \v},\p\oplus\b)$ as the
semidirect product:
\begin{equation*}
\begin{split}
[(x,u),(y,v)]_{\g\ltimes_\r \v}  &:= ([x,y]\xg,\r(x)(v)-\r(y)(u)), \\
(\p\oplus\b)(x,u) &:=  (\p(x),\b(u)).
\end{split}
\end{equation*}

\begin{defi}
   Let  $\hlg$  be a  Hom-Lie algebra. Two representations
   $(\v,\b,\r)$ and $(\v',\b',\r')$ are called {\bf equivalent} if
   there exists a linear isomorphism $\varphi:V\rightarrow V'$
   such that
\begin{equation}
\varphi\rho(x)=\rho'(x)\varphi,\;\;\beta'\varphi(x)=\varphi\beta(x),\forall
x\in \g.
\end{equation}
\end{defi}

Let $V_1$ and $V_2$ be two vector spaces. For a linear map
$\phi:V_1\longrightarrow V_2$, we define the map
$\phi^*:V_2^*\longrightarrow V_1^*$ by
\begin{equation}\label{eq:dual1}
\la\phi^*(\xi),v\ra=\la \xi,\phi(v)\ra, \forall\ v\in V_1,\xi\in
V_2^*.
\end{equation}

Given a representation \rep\  of a Hom-Lie algebra $\hlg$, define
$\rs:\g\lr\gl(\v^*)$ by
\begin{equation}
\la\rs(x)(\xi),v\ra=-\la \xi,\r(x)v\ra,\ \forall\
x\in\g,v\in\v,\xi\in\v^*.
\end{equation}
As noted in~\cite{BM}, in general $(\v^*,\beta^*,\rs)$ is not a representation of $\g$ on $\v^*$ with respect to $\beta^*$.

We now introduce a Hom-variation of dual representation.
Let  $\hlg$  be a  Hom-Lie algebra and let  $(\v,\b,\r)$ be a representation of $\hlg$.
Define
\begin{equation}
\ro:=\rs\,  \p: \g\lr\gl(\v^*).
\label{eq:ro}
\end{equation}
In other words,

\begin{equation}
\la\ro(x)(\xi),v\ra=-\la \xi,\r(\p(x))v\ra,\ \forall\
x\in\g,v\in\v,\xi\in\v^*.
\end{equation}
In general, $(\v^*,\b^*,\ro)$ is not a representation of $\g$. However, by the definition of a representation of $\g$, we have

\begin{lem}
    Let \hlg\ be a Hom-Lie algebra and \rep\ be a representation. Then \drep\ is
    a representation if and only if the following equations hold:
    \begin{itemize}
        \item[(i)]  $\b\, \r(x)=\b\, \r(\p^2(x))$;
        \item[(ii)]
        $\r(\p^2[x,y]\xg)\, \b=\r(\p(x))\, \r(\p^2(y))-\r(\p(y))\, \r(\p^2(x)),\;\;\forall
        x,y\in\g$.
    \end{itemize}
In this case, \drep\ is called the {\bf Hom-dual representation} of \rep.
\end{lem}

As can be expected for a dual representation, $\ro$ should satisfy
$(\ro)^\circ=\r$, that is,  $\r(\p^2(x))=\r(x)$ for all $x\in \g$.
Moreover, the following obvious result further highlights the relevance of the latter condition to representations.

\begin{lem}
Let $\hlg$ be a Hom-Lie algebra and \rep\ be a representation.
\begin{enumerate}
\item If $\r(\p^2(x))=\r(x)$ for any $x\in\g$, then \drep\ is
    a representation of $\g$.
\item
If $\b$ is invertible and \drep\ is a representation of $\g$, then we have $\r(\p^2(x))=\r(x)$.
\end{enumerate}
 \label{lem:wi}
\end{lem}

Motivated by Lemma~\ref{lem:wi}, a representation \rep\ is called {\bf weakly involutive} if $\r(\p^2(x))=\r(x)$. By the definitions of $\ro$ and $\rho^*$ we obtain

\begin{lem}\label{dual rep, weakly}
Let $\hlg$ be a Hom-Lie algebra and \rep\ be a weakly involutive
representation. Then the Hom-dual representation \drep\ is also a weakly involutive representation.
\end{lem}

Focusing on the adjoint representation, we give
\begin{defi}
    A Hom-Lie algebra \hlg\ is called {\bf weakly involutive} if the adjoint representation $(\g,\p,\ad)$ is weakly involutive, more precisely,  $[\p^2(x),y]\xg=[x,y]\xg$ for all $x,y\in\g$.
\end{defi}

\begin{cor}\label{cor add}
Let \hlg\ be a weakly involutive Hom-Lie algebra. Then we have
    \begin{equation*}
    (\adx^*{\ps}^2)( a)=\adx^*(a),\;\;\forall x\in \g, a\in \g^*.
    \end{equation*}
    In particular, we have  $ (\adox{\ps}^2)( a)=\adox (a)$.
\end{cor}

\begin{proof}
    Since $\g$ is weakly involutive, we have $[\p^2(x),y]\xg=[x,y]\xg$. Then we have
    \begin{equation*}
    \begin{split}
    \la (\adx^*{\ps}^2)(  a),y\ra=-\la a,\p^2[x,y]\xg\ra=-\la a,[x,y]\xg\ra=\la \adx^* (a),y\ra.
    \end{split}
    \end{equation*}
    Therefore the conclusion holds.
\end{proof}

It is straightforward to get the following conclusion.

\begin{pro}\label{semidirect, weakly}
    Let $\hlg$ be a Hom-Lie algebra and \rep\ be a representation. Then the Hom-Lie algebra $\g\ltimes_\r
    \v$ is weakly involutive if and only if
    \begin{itemize}
        \item[(i)]  $\hlg$ is weakly involutive;
        \item[(ii)] \rep\ is weakly involutive;
        \item[(iii)] $\rho(x)\,  \b^2=\rho(x)$ for all $x\in\g$.
    \end{itemize}
\end{pro}
\begin{cor}\label{co:semidirect}
Let $\hlg$ be a Hom-Lie algebra and \rep\ be a weakly involutive representation. Then the Hom-Lie algebra $\g\ltimes_{\ro}
\v^*$ is weakly involutive if and only if
\begin{itemize}
    \item[(i)]  $\hlg$ is weakly involutive;
    \item[(ii)] $\rho(\p(x))\,  \b^2=\rho(\p(x))$ for all $x\in\g$.
\end{itemize}
\end{cor}
\begin{proof}
Since \rep\ is weakly involutive,  \drep\ is weakly involutive by Lemma~\ref{dual rep, weakly}. Moreover $\ro(x)\,  (\b^*)^2=\ro(x)$ for all $x\in\g$ if
and only if  $\rho(\p(x))\,  \b^2=\rho(\p(x))$ for all $x\in\g$. Hence the conclusion holds by Proposition \ref{semidirect, weakly}.
\end{proof}

\begin{cor}
Let $\hlg$ be a Hom-Lie algebra, \rep\ be a weakly involutive representation and $\g\ltimes_\r \v$ be a weakly involutive Hom-Lie algebra. Then the Hom-Lie algebra $\g\ltimes_{\ro} \v^*$ is weakly involutive.
\end{cor}

\begin{defi}
    A bilinear form $\bb(\cdot,\cdot)$ on a Hom-Lie algebra \hlg\ is called {\bf invariant} if
    \begin{equation*}
    \bb([x,y]\xg,z)= \bb(x,[\p(y),z]\xg),\ \ \bb(\p(x),y)= \bb(x,\p(y)),\ \ \forall\ x,y,z\in\g.
    \end{equation*}
\end{defi}

\begin{rem}
If $\p$ is invertible, then the above invariant conditions are equivalent
to
$$\bb([x,y]\xg,\p(z))=\bb(\p(x),[y,z]\xg),\ \bb(\p(x),y)=\bb(x,\p(y)).$$
\end{rem}
The following proposition gives the close relationship between weakly involutive condition on a Hom-Lie algebra and the existence of a nondegenerate invariant bilinear form on the Hom-Lie algebra.

\begin{pro}\label{invariant and weakly involutive}
    Let \hlg\ be a Hom-Lie algebra with a nondegenerate symmetric invariant bilinear form $\bb(\cdot,\cdot)$. Then $\g$ is
    a weakly involutive Hom-Lie algebra,  and $(\g,\p,\ad)$ and
$(\g^*,\p^*,\ado)$ are equivalent
representations of \hlg\,. Conversely, if \hlg\
is a weakly involutive Hom-Lie algebra,  and $(\g,\p,\ad)$ and
$(\g^*,\p^*,\ado)$ are equivalent representations of  \hlg\,, then there exists a
nondegenerate invariant bilinear form $\bb(\cdot,\cdot)$ on \hlg\,.
\end{pro}

Note that the converse statement does not include the symmetric condition.

\begin{proof}
Let \hlg\ be a Hom-Lie algebra with a nondegenerate symmetric invariant bilinear form $\bb(\cdot,\cdot)$. For $x,y,z\in\g$, we have
    \begin{equation*}
    \begin{split}
    \bb(x,[y,z]\xg)=-\bb([z,y]\xg,x)=-\bb(z,[\p(y),x]\xg)=\bb([x,\p(y)]\xg,z).
    \end{split}
    \end{equation*}
    Then we get
    \begin{equation*}
    \begin{split}
    \bb([x,y]\xg,z)=\bb(x,[\p(y),z]\xg)=\bb([x,\p^2(y)]\xg,z).
    \end{split}
    \end{equation*}
    Since $\bb(\cdot,\cdot)$ is nondegenerate, we get $[x,y]\xg=[x,\p^2(y)]\xg$.
    Moreover, define a linear transformation $\varphi:\g\rightarrow
    \g^*$ by
$$\varphi(x) (y):=\langle \varphi(x), y\rangle=B(x,y),\;\;\forall
x,y\in \g.$$
Then $\varphi$ is a linear isomorphism and, for all
$x,y,z\in \g$, we have
\begin{eqnarray*}
\varphi(\ad(x)(y))(z)&=&\langle \varphi([x,y]),z\rangle
=B([x,y],z)=-B([y,x],z) =-B(y, [\p(x),z])\\
& =&-\langle \varphi(y), [\p(x),z])\rangle=\langle
\ado(x)(\varphi(y)), z\rangle=\ado(x)\varphi(y)(z);
\\
\varphi\p(x)(y)&=&\langle \varphi\p(x),
y\rangle=B(\p(x), y)=B(x,\p(y))=\langle
\varphi(x),\p(y)\rangle\\
&=&\langle \p^*\varphi(x),y\rangle=\p^*\varphi(x)(y).
\end{eqnarray*}
Hence $(\g,\p,\ad)$ and
$(\g^*,\p^*,\ado)$ are equivalent representations of \hlg\,.

Conversely, let \hlg\ be
    a weakly involutive Hom-Lie algebra. If $(\g,\p,\ad)$ and
$(\g^*,\p^*,\ado)$ are equivalent
representations of  \hlg\,, then there exists a
linear isomorphism $\psi:\g\rightarrow \gs$ satisfying
$$\psi\ad(x)=\ado(x)\psi,\;\;\psi\p(x)=\p^*\psi(x),\;\;\forall x\in\g.$$
Define a bilinear form on
$\g$ by
$$B(x,y):=\psi(x) y=\langle \psi(x), y\rangle,\;\;\forall
x,y\in \g.$$
Then $B$ is nondegenerate and, by a similar argument, one can
show that $B$ is invariant.
\end{proof}

\section{Matched pairs, Manin triples and Hom-Lie bialgebras}
\label{sec:match}

We first recall the notion of matched pairs of Hom-Lie algebras.
Let \hlg\ and $(\g\pr,[\cdot,\cdot]_{\g\pr},\phi_{_{\g\pr}})$ be two
Hom-Lie algebras. Let $\r:\g\lr\gl(\g\pr)$ and
$\r\pr:\g\pr\lr\gl(\g)$ be two linear maps. On the direct sum $\g\oplus\g\pr$ of the underlying vector spaces, define
\begin{equation}
\begin{split}
\phi_{_{d}}:\g\oplus\g\pr\lr\g\oplus\g\pr, \quad
\phi_{_{d}}(x,x{\pr}):=(\p(x),\phi_{_{\g\pr}}(x{\pr})),
\end{split}
\end{equation}
and define a skew-symmetric bilinear map
$[\cdot,\cdot]_d:\t^2(\g\oplus\g\pr)\lr\g\oplus\g\pr$ by
\begin{equation}\label{db}
\begin{split}
[(x,x\pr),(y,y\pr)]_d:=\big([x,y]\xg-\r\pr(y\pr)(x)+\r\pr(x\pr)(y),[x\pr,y\pr]_{\g\pr}+\r(x)(y\pr)-\r(y)(x\pr)\big).
\end{split}
\end{equation}

\begin{thm}[\cite{BS}]\label{direct-Hom-lie}
    Let \hlg\ and $(\g\pr,[\cdot,\cdot]_{\g\pr},\phi_{_{\g\pr}})$ be two
    Hom-Lie algebras. Then $(\g\oplus\g\pr,[\cdot,\cdot]_d,\phi_{_{d}})$ is a Hom-Lie algebra if and only if $(\g', \phi_{_{\g\pr}},\rho)$ and $(\g, \p,\rho')$
     are
    representations of \hlg\ and $(\g\pr,[\cdot,\cdot]_{\g\pr},\phi_{_{\g\pr}})$ respectively
    and
    \begin{equation}\label{mp1}
    \begin{split}
    \r\pr(\pp(x\pr))[x,y]\xg=& [\r\pr(x\pr)(x),\p(y)]\xg+[\p(x),\r\pr(x\pr)(y)]\xg+\r\pr(\r(y)(x\pr))(\p(x)) \\
    &-\r\pr(\r(x)(x\pr))(\p(y)),
    \end{split}
    \end{equation}
    \begin{equation}\label{mp2}
    \begin{split}
    \r(\p(x))[x\pr,y\pr]_{\g \pr}=& [\r(x)(x\pr),\pp(y\pr)]_{\g \pr}+[\pp(x\pr),\r(x)(y\pr)]_{\g \pr}+\r(\r\pr(y\pr)(x))(\pp(x\pr)) \\
    &-\r(\r\pr(x\pr)(x))(\pp(y\pr)).
    \end{split}
    \end{equation}
\end{thm}

\begin{defi}[\cite{BS}]
    A {\bf matched pair of Hom-Lie algebras} is a quadruple $(\g,\g\pr;\r,\r\pr)$ consisting of two Hom-Lie algebras \hlg\ and $(\g\pr,[\cdot,\cdot]_{\g\pr},\phi_{_{\g\pr}})$, together with representations $\r:\g\lr\gl(\g\pr)$ and $\r\pr:\g\pr\lr\gl(\g)$ with respect to $\pp$ and $\p$ respectively, such that the compatibility conditions (\ref{mp1}) and (\ref{mp2}) are satisfied.
\end{defi}

It is straightforward to reach the following conclusion.
\begin{pro}\label{direct product, weakly}
Let $(\g,\g\pr;\r,\r\pr)$ be a matched pair of Hom-Lie algebras. Then the Hom-Lie algebra $(\g\oplus\g\pr,[\cdot,\cdot]_d,\phi_{_{d}})$ is weakly involutive if and only if the following conditions hold:
\begin{itemize}
    \item[(i)]  $\hlg$ is weakly involutive and $(\g', \phi_{_{\g\pr}},\rho)$ is a weakly involutive representation of $\g$;
    \item[(ii)]  $(\g\pr,[\cdot,\cdot]_{\g\pr},\phi_{_{\g\pr}})$ is weakly involutive and $(\g, \p,\rho')$ is a weakly involutive representation of $\g\pr$;
        \item[(iii)] $\rho(x)\,  \phi_{_{\g\pr}}^2=\rho(x)$ for all $x\in\g$;
        \item[(iv)] $\r\pr(x\pr)\,  \p^2=\r\pr(x\pr)$ for all $x\pr\in\g\pr$.
\end{itemize}
\end{pro}

We now introduce the closely related notion of Manin triples of Hom-Lie algebras.
\begin{defi}
    A {\bf Manin triple of Hom-Lie algebras} is a triple of Hom-Lie algebras $(\k;\g,\g\pr)$ together with a nondegenerate symmetric invariant bilinear form $\bb(\cdot,\cdot)$ on $\k$ such that $\g$ and $\g\pr$ are isotropic Hom-Lie subalgebras of $\k$ and $\k=\g\oplus\g\pr$ as vector spaces.
\end{defi}

By Proposition~\ref{invariant and weakly involutive}, we have
\begin{cor}
    Let $(\k;\g,\g\pr)$ be a Manin triple of Hom-Lie algebras. Then $\k$ and hence $\g, \g\pr$ are weakly involutive Hom-Lie algebras.
\end{cor}

\begin{pro}\label{Hom-Lie bi,Manin triple}
Let \hlg\ and $(\g^*,[\cdot,\cdot]_{\g^*},\phi_{_{\gs}})$ be two weakly involutive Hom-Lie algebras. Then
$(\g\oplus\gs;\g,\gs)$ is a Manin triple of Hom-Lie algebras
associated to the nondegenerate symmetric invariant bilinear form
on $\g\oplus\gs$ defined by
\begin{equation}\label{bilinear form}
\bb(x+a,y+b):=\la x,b\ra+\la y,a\ra
\end{equation}
if and only if $\phi_{_{\gs}}=\ps$ and $(\g,\gs;\ado,\dao)$ is a matched pair of Hom-Lie
algebras. Here $\dao$ is the Hom-dual representation of the adjoint representation $\da$ of the Hom-Lie algebra $\g^*$.
\end{pro}

\begin{proof}  Suppose that $\phi_{_{\gs}}=\p^*$ and $(\g,\gs;\ado,\dao)$ is a matched pair of Hom-Lie
algebras. Then $(\g\oplus\gs,[\cdot,\cdot]_d,\p\oplus\ps)$ is a
Hom-Lie algebra, where $[\cdot,\cdot]_d$ is given by
\begin{equation}\label{d-bracket}
[x+a,y+b]_d:=[x,y]\xg+\adox b-\ado_y a+[a,b]_{\gs}+\dao_{a}
y-\dao_{b} x.
\end{equation}

Let $x,y\in\g$ and $a,b\in\gs$. We just need to prove that the
bilinear form defined by Eq.~(\ref{bilinear form}) is invariant.
Setting $\phi_{_d}=\p\oplus\ps$, we have
\begin{align*}
\bb(\pd(x),a)&=\la \p(x),a\ra=\la x,\ps(a)\ra=\bb(x,\pd(a)),\\
\bb(x,[\pd(y),a]_d)&=\la x,\ado_{\p(y)} a\ra=-\la [\p^2(y),x]\xg,a\ra=\la [x,y]\xg,a\ra=\bb([x,y]_d,a),\\
\bb([x,a]_d,y)&=\la \adox a,y\ra=\la  [y,\p(x)]\xg,a\ra,\\
\bb(x,[\pd(a),y]_d)&=-\la x,\ado_{y} \ps(a)\ra=\la [\p(y),x]\xg,\ps(a)\ra=\la  [y,\p(x)]\xg,a\ra=\bb([x,a]_d,y),\\
\bb([x,a]_d,b)&=-\la \dao_{a} x,b\ra=\la
x,[\ps(a),b]_{\gs}\ra=\bb(x,[\pd(a),b]_d).
\end{align*}
Therefore the bilinear form defined by Eq.~(\ref{bilinear form})
is invariant.

Conversely, let $(\g\oplus\gs;\g,\gs)$ be a Manin triple of
Hom-Lie algebras associated to the invariant bilinear form $\bb$
given by Eq.~(\ref{bilinear form}). Then for any $x,y\in\g$ and
$a,b\in\gs$, due to the invariance of $\bb$, we have
\begin{equation*}
\begin{split}
\bb(\p(x),a)=&\la \p(x),a\ra=\la x,\ps(a)\ra,\\
\bb(x,\phi_{_{\gs}}(a))=&\la x,\phi_{_{\gs}}(a)\ra,
\end{split}
\end{equation*}
which implies $\phi_{_{\gs}}=\ps$. Since
\begin{equation*}
\begin{split}
\bb([x,a]_d,y)=& \bb([y,\p(x)]\xg,a)=\la -\ad_{\p(x)} y,a\ra=\la \adox a,y\ra,\\
\bb([x,a]_d,b)=& \bb(x,[\ps(a),b]_{\gs})=\la x,\da_{\ps(a)}
b\ra=\la -\dao_{a} x,b\ra,
\end{split}
\end{equation*}
we have
\begin{equation*}
[x,a]_d=\adox a-\dao_{a} x.
\end{equation*}
So the Hom-Lie bracket on $\g\oplus\gs$ is given by Eq.~(\ref{d-bracket}). Therefore, $(\g,\gs;\ado,\dao)$ is a matched
pair of Hom-Lie algebras.
\end{proof}

For a Hom-Lie algebra \hlg\ (resp.
$(\gs,[\cdot,\cdot]_{\gs},\ps)$), let $\d_{\g^*}:\gs\lr\t^2\gs$
(resp. $\d_{\g}:\g\lr\t^2\g$) be the dual map of $\bra:\t^2
\g\lr\g$ (resp. $[\cdot,\cdot]_{\gs}:\t^2\gs\lr\gs$), i.e.,
\begin{equation}\label{eq:dual} \la \d_{\g^*}( a),x\t y\ra=\la a,[x,y]\xg\ra,\ \ \ \la\d_\g(x), a\t b\ra=\la x,[ a, b]_{\gs}\ra.\end{equation}
In particular, we set $\Delta:=\d_\g$.

\begin{thm}\label{Hom-lie-bi,matched pair}
    Let \hlg\ and \hlgs\ be two weakly involutive Hom-Lie algebras.  Then $(\g,\gs;\ado,\dao)$ is a matched pair of Hom-Lie
algebras if and only if
\begin{eqnarray}
    \label{Hom-Lie-bi-c1} \d[x,y]\xg &=& \ad_{\p(x)}\d(y)-\ad_{\p(y)}\d(x),
    \end{eqnarray}
where $\Delta=:\d_\g$ is given by Eq.~(\ref{eq:dual}).
\end{thm}

\begin{proof}
    By Theorem \ref{direct-Hom-lie}, $(\g,\gs;\ado,\dao)$  is a
    matched pair of Hom-Lie algebras  if and only if
    \begin{eqnarray}
    \label{c1}  \dao_{\ps( a)}[x,y]\xg &=& [\dao_ a x,\p(y)]\xg,+[\p(x),\dao_ a y]\xg+\dao_{\ado_y  a}\p(x)-\dao_{\ado_x  a}\p(y), \\
    \label{c2} \ado_{\p(x)}[ a, b] _{\gs}&=& [\ado_x  a,\ps( b)]_{\gs}+[\ps( a),\ado_x  b]_{\gs}+\ado_{\dao_ b x}\ps( a)-\ado_{\dao_ a x}\ps( b).
    \end{eqnarray}
    Then we get
    \begin{align*}
    0=& \la -\dao_{\ps( a)}[x,y]\xg+[\dao_ a x,\p(y)]\xg,+[\p(x),\dao_ a y]\xg+\dao_{\ado_y  a}\p(x)-\dao_{\ado_x  a}\p(y), b\ra \\
    =& \la [x,y]\xg,[{\ps}^2( a), b]_{\gs}\ra-\la \ad_{\p(y)}\dao_ a x, b\ra+\la \ad_{\p(x)}\dao_ a y, b\ra-\la \p(x),[\ps\ado_y  a, b]_{\gs}\ra \\
    & +\la \p(y),[\ps\ado_x  a, b]_{\gs}\ra \\
    =& \la [x,y]\xg,[ a, b]_{\gs}\ra-\la x,[\ps( a),\ado_y  b]_{\gs}\ra+\la y,[\ps( a),\ado_x  b]_{\gs}\ra-\la x,[{\ps}^2\ado_y  a,\ps b]_{\gs}\ra  \\
    & +\la y,[{\ps}^2\ado_x  a,\ps b]_{\gs}\ra\\
    =& \la [x,y]\xg,[ a, b]_{\gs}\ra-\la x,[\ps( a),\ado_y  b]_{\gs}\ra+\la y,[\ps( a),\ado_x  b]_{\gs}\ra-\la x,[\ado_y  a,\ps b]_{\gs}\ra  \\
    & +\la y,[\ado_x  a,\ps b]_{\gs}\ra\\
    =& \la \d[x,y]\xg, a\t b\ra-\la \d(x),\ps( a)\t\ado_y b \ra+\la \d(y),\ps( a)\t\ado_x  b\ra-\la \d(x),\ado_y  a\t\ps b\ra  \\
    & +\la \d(y),\ado_x  a\t\ps b\ra,\\
    =& \la \d[x,y]\xg, a\t b\ra-\la \d(x),\ads_{\p(y)}( a\t b) \ra+\la \d(y),\ads_{\p(x)}( a\t b) \ra\\
    =& \la \d[x,y]\xg, a\t b\ra+\la \ad_{\p(y)}\d(x), a\t b\ra-\la \ad_{\p(x)}\d(y), a\t b\ra,
    \end{align*}
    which is exactly Eq.~(\ref{Hom-Lie-bi-c1}). Similarly, we could deduce that Eq.~(\ref{c2}) is equivalent to the equation
    \begin{equation}\label{Hom-Lie-bi-c2}
    \d_{\gs}[ a, b]_{\gs}=\da_{\ps( a)}\d_{\gs}( b)-\da_{\ps( b)}\d_{\gs}( a).
    \end{equation}
    Next we prove that Eq.~(\ref{Hom-Lie-bi-c1}) and Eq.~(\ref{Hom-Lie-bi-c2}) are equivalent.
    In fact,
    we have
    \begin{align*}
    & \la \d[x,y]\xg-\ad_{\p(x)}\d(y)+\ad_{\p(y)}\d(x),a\t b\ra\\
    =& \la [x,y]\xg,[a,b]_{\gs}\ra+\la \d(y),\adox(a\t b)\ra-\la \d(x),\ado_y (a\t b)\ra\\
    =& \la \d_{\gs} [a,b]_{\gs},x\t y\ra+\la \d(y),\adox (a)\t \ps(b)+\ps(a)\t \adox (b)\ra\\
    & -\la \d(x),\ado_y (a)\t \ps(b)+\ps(a)\t \ado_y (b)\ra\\
    =& \la \d_{\gs} [a,b]_{\gs},x\t y\ra+\la y,[\adox (a),\ps(b)]_{\gs}\ra+\la y,[\ps(a),\adox (b)]_{\gs}\ra\\
    & -\la x,[\ado_y (a),\ps(b)]_{\gs}\ra-\la x,[\ps(a),\ado_y (b)]_{\gs}\ra\\
    =& \la \d_{\gs} [a,b]_{\gs},x\t y\ra+\la \dao_b (y),\adox (a)\ra-\la \dao_a (y),\adox (b)\ra-\la \dao_b (x),\ado_y (a)\ra\\
    &+\la \dao_a (x),\ado_y (b)\ra\\
    =& \la \d_{\gs} [a,b]_{\gs},x\t y\ra-\la a,[\p(x),\dao_b (y)]\xg\ra+\la b,[\p(x),\dao_a (y)]\xg\ra\\
    & +\la a,[\p(y),\dao_b (x)]\xg\ra-\la b,[\p(y),\dao_a (x)]\xg\ra\\
    =& \la \d_{\gs}[ a, b]_{\gs}-\da_{\ps( a)}\d_{\gs}( b)+\da_{\ps( b)}\d_{\gs}( a),x\t y\ra.
    \end{align*}
    Therefore the two equations are equivalent.
This gives what we need.
\end{proof}

\begin{defi}\label{defi Hom-Lie bialgebra}
    A pair of weakly involutive Hom-Lie algebras \hlg\ and \hlgs\ with  $\Delta=:\d_\g$ given by Eq.~(\ref{eq:dual}) is called a {\bf Hom-Lie bialgebra} if Eq.~(\ref{Hom-Lie-bi-c1}) holds. We denote it by $(\g,\g^*)$ or $(\g,\Delta)$.
\end{defi}

\begin{rem}
We note that this notion of a Hom-Lie bialgebra is different from either of the three notions of a
Hom-Lie bialgebra given in \cite{DY-hlbi,BS,CS-purely}. Even
though it has the same compatibility condition as in
\cite{DY-hlbi}, it is just for weakly involutive Hom-Lie algebras,
not all Hom-Lie algebras. Under this compatibility condition,
there is a natural Hom-Lie algebra structure on $\g\oplus\gs$ as
follows, which does not exist in \cite{DY-hlbi}.
\end{rem}

\begin{cor}\label{bialgebra, weakly}
Let $(\g,\g^*)$ be a Hom-Lie bialgebra. Then the Hom-Lie algebra $(\g\oplus\gs,[\cdot,\cdot]_d,\p\oplus\ps)$ is weakly involutive.
\end{cor}
\begin{proof}
Since $(\g,\g^*)$ is a Hom-Lie bialgebra, it follows that \hlg\ and \hlgs\ are weakly involutive Hom-Lie algebras, and $(\g,\gs;\ado,\dao)$ is a matched pair of Hom-Lie algebras. Then $(\g,\p,\ad)$ and $(\gs,\ps,\da)$ are weakly involutive representations of $\g$ and $\gs$ respectively. Hence by Lemma \ref{dual rep, weakly}, $(\gs,\ps,\ado)$ and $(\g,\p,\dao)$ are weakly involutive representations of $\g$ and $\gs$ respectively. By Corollary \ref{cor add}, we have
\begin{equation*}
\ado_x\,  {\ps}^2=\ado_x,\ \ \dao_a\,  \p^2=\dao_a,
\end{equation*}
for any $x\in\g,a\in\gs$. Finally, Proposition \ref{direct product, weakly} gives the conclusion.
\end{proof}

\begin{defi}\label{defi. of homomorphism of Hom-Lie bialgebras}
A homomorphism of Hom-Lie bialgebras $f:(\g,\d_{\g})\lr (\k,\d_{\k})$ is a homomorphism of Hom-Lie algebras such that
$$(f\t f)\,  \d_{\g}=\d_{\k}\,  f.$$
\end{defi}

Combining Proposition \ref{Hom-Lie bi,Manin triple}, Theorem \ref{Hom-lie-bi,matched pair} and Definition~\ref{defi Hom-Lie bialgebra}, we arrive at the following conclusion which, when $\p$ is the identity, recovers the classical characterization of Lie bialgebras in terms of matched pairs and Manin triples.

\begin{thm}
    Let \hlg\ and \hlgs\ be two weakly involutive Hom-Lie algebras. Then the following     conditions are equivalent.
    \begin{itemize}
        \item[(i)]  $(\g,\gs)$ is a Hom-Lie bialgebra.
        \item[(ii)]  $(\g,\gs;\ado,\dao)$ is a matched pair of Hom-Lie algebras.
        \item[(iii)]  $(\g\oplus\gs;\g,\gs)$ is a Manin triple of Hom-Lie algebras associated to the invariant bilinear form given by Eq. (\ref{bilinear form}).
    \end{itemize}
\label{thm:equiv}
\end{thm}

\section{Coboundary Hom-Lie bialgebras and the classical Hom-Yang-Baxter equation}
\label{sec:quasi}

In this section, we study coboundary Hom-Lie bialgebras and their
relationship with the classical Hom-Yang-Baxter equation. For a weakly involutive Hom-Lie algebra
\hlg\ and $r\in\g\otimes\g$, define linear maps
\begin{equation}\label{definition of d}
\d:\g\lr\g\otimes \g, \quad \d(x)=\ad_x
r(:=[x,r]\xg):=(\p\t\ad_x+\ad_x\t\p)r, \forall x\in \g,
\end{equation}
and
\begin{equation}\label{dual bracket}
[\cdot,\cdot]_{\gs}:\gs\otimes \gs\lr\gs, \quad \la [ a, b]_{\gs},x\ra=\la\d(x), a\t b\ra, \forall a, b\in \gs, x\in \g.
\end{equation}

In order for \hlgs\ to be a weakly involutive Hom-Lie algebra and for $(\g,\gs)$
to be a Hom-Lie bialgebra, the following conditions should hold.
\begin{enumerate}
\item  $\ps$ is a Lie algebra homomorphism with respect to
$[\cdot,\cdot]_{\gs}$, or equivalently,
$\d(\p(x))=(\p\t\p)\d(x)$;
\item
$[{\ps}^2(a),b]_{\gs}=[a,b]_{\gs}$, or equivalently,
$(\p^2\t\id)\d(x)=\d(x)$;
\item
$\d[x,y]\xg-(\ad_{\p(x)}\d(y)-\ad_{\p(y)}\d(x))=0$;
\item  \hlgs\ is a Hom-Lie algebra.
\end{enumerate}
So we first investigate these conditions. For the first three conditions we have the following result, noting the common factor $(\p\t\id-\id\t\p)r$ on the right hand sides.

\begin{lem}\label{lem:delta}
Let \hlg\ be a weakly involutive Hom-Lie algebra and $r\in\g\otimes
\g$. Define a linear map $\d:\g\lr\g\otimes \g$ by
Eq.~(\ref{definition of d}). Then we have
\begin{enumerate}
\item
$\d(\p(x))-(\p\t\p)\d(x)= (\ad_{\p(x)}\, \p\t \p-\p\t \ad_{\p(x)}\, \p)(\p\t\id-\id\t\p)r;
$
\label{it:alghom}
\item
$    (\p^2\t\id)\d(x)-\d(x)= (\p\t \ad_x)(\p\t\id+\id\t\p)(\p\t\id-\id\t\p)r;
$
\label{it:inv}
\item
$\d[x,y]\xg-(\ad_{\p(x)}\d(y)-\ad_{\p(y)}\d(x))\\
    = (\ad_{[x,y]\xg}\, \p\t \p-\p\t \ad_{[x,y]\xg}\, \p)(\p\t\id-\id\t\p)r.$
\label{it:comp}
\end{enumerate}
\end{lem}

\begin{proof}
(\mref{it:alghom}).  For $x,y\in\g$, we have
    \begin{equation*}
    \begin{split}
    \d(\p(x))=& [\p(x),r]\xg=\ad_{\p(x)} r\\
    =& (\ad_{\p(x)}\t \p)r+(\p\t \ad_{\p(x)})r\\
    =& (\ad_{\p(x)}\, \p^2\t \p)r+(\p\t \ad_{\p(x)}\, \p^2)r\\
    =& (\ad_{\p(x)}\, \p\t \p)(\p\t\id)r+(\p\t \ad_{\p(x)}\, \p)(\id\t\p)r
    \end{split}
    \end{equation*}
    and
    \begin{equation*}
    \begin{split}
    (\p\t\p)\d(x)=&(\p\t\p)\ad_x r\\
    =& (\p\t\p)(\ad_x\t \p)r+(\p\t\p)(\p\t \ad_x)r\\
    =& (\ad_{\p(x)}\, \p\t \p^2)r+(\p^2\t \ad_{\p(x)}\, \p)r\\
    =& (\ad_{\p(x)}\, \p\t \p)(\id\t\p)r+(\p\t \ad_{\p(x)}\, \p)(\p\t\id)r.
    \end{split}
    \end{equation*}
    Thus we get
    \begin{equation*}
    \begin{split}
    \d(\p(x))-(\p\t\p)\d(x)=& (\ad_{\p(x)}\, \p\t \p-\p\t \ad_{\p(x)}\, \p)(\p\t\id-\id\t\p)r.
    \end{split}
    \end{equation*}

(\ref{it:inv}). For $x,y\in\g$, we have
    \begin{align*}
    (\p^2\t\id)\d(x)=& (\p^2\t\id)(\ad_{x}\t \p)r+(\p^2\t\id)(\p\t \ad_{x})r\\
    =& (\ad_{x}\t \p)r+(\p^3\t \ad_{x})r
    \end{align*}
    and
    \begin{align*}
\d(x)=&(\ad_x\t \p)r+(\p\t \ad_x)r.
    \end{align*}
Then we have
    \begin{align*}
    (\p^2\t\id)\d(x)-\d(x)=& (\p^3\t \ad_{x})r-(\p\t \ad_x)r= (\p^3\t \ad_{x})r-(\p\t \ad_x\, \p^2)r\\
    =& (\p\t \ad_x)(\p^2\t\id-\id\t\p^2)r\\
    =& (\p\t \ad_x)(\p\t\id+\id\t\p)(\p\t\id-\id\t\p)r.
    \end{align*}

(\ref{it:comp}).
For $x,y\in\g$, we have
    \begin{align*}
    \d[x,y]\xg=& [[x,y]\xg,r]\xg=\ad_{[x,y]\xg}r\\
    =& (\ad_{[x,y]\xg}\t \p)r+(\p\t \ad_{[x,y]\xg})r\\
    =& (\ad_{[x,y]\xg}\, \p^2\t \p)r+(\p\t \ad_{[x,y]\xg}\, \p^2)r\\
    =& (\ad_{[x,y]\xg}\, \p\t \p)(\p\t\id)r+(\p\t \ad_{[x,y]\xg}\, \p)(\id\t\p)r
    \end{align*}
    and
    \begin{align*}
    & \ad_{\p(x)}\d(y)-\ad_{\p(y)}\d(x)\\
    =& (\ad_{\p(x)}\t \p+\p\t \ad_{\p(x)})(\ad_y \t \p+\p\t\ad_y)r\\
    &- (\ad_{\p(y)}\t \p+\p\t \ad_{\p(y)})(\ad_x \t \p+\p\t\ad_x)r\\
    =& (\ad_{\p(x)}\, \ad_y\t \p^2)r-(\ad_{\p(y)}\, \ad_x\t \p^2)r+(\p^2\t\ad_{\p(x)}\, \ad_y)r\\
    &-(\p^2\t\ad_{\p(y)}\, \ad_x)r\\
    =& (\ad_{[x,y]\xg}\, \p\t \p^2)r+(\p^2\t\ad_{[x,y]\xg}\, \p)r\\
    =& (\ad_{[x,y]\xg}\, \p\t \p)(\id\t\p)r+(\p\t \ad_{[x,y]\xg}\, \p)(\p\t\id)r.
    \end{align*}
This gives the desired equation.
\end{proof}

Therefore, to give the definition of coboundary Hom-Lie bialgebras,
requiring $(\p\t\id)r=(\id\t\p)r$ is a natural choice. This leads us to the following definition.

\begin{defi}
    A {\bf coboundary Hom-Lie bialgebra} is a Hom-Lie bialgebra $(\g,\d)$ such that the linear map $\d:\g\rightarrow \g\otimes\g$ is given
    by Eq.~(\ref{definition of d}),
    where $r\in\g\otimes \g$ satisfies
    \begin{equation}\label{ra*=ar}
    (\p\t\id)r=(\id\t\p)r.
    \end{equation}
\end{defi}

For a Hom-Lie algebra \hlg\ and $r=\xqh{i}\x i\t\y i\in\g\otimes
\g$, define $[r,r]\xg\in\t^3\g$ by
\begin{equation*}
\begin{split}
[r,r]\xg:=&\xqh{i,j} ([\x i,\x j]\xg\t\p(\y i)\t\p(\y j)+\p(\x
i)\t[\y i,\x j]\xg\t\p(\y j)+\p(\x i)\t\p(\x j)\t[\y i,\y j]\xg)
\end{split}
\end{equation*}
and set $\sigma(r):=\xqh{i}\y i\t\x i$.

For any linear map $\d:\g\lr\g\otimes \g$ and any $x\in\g$, let
    \begin{equation*}
    \jac_{_{\d}}(x)=\xqh{c.p.}(\p\t\d)\d(x),
    \end{equation*}
where `$\xqh{c.p.}$' means that the sum is taken of the term after the
summation sign and together with the two similar terms obtained by cyclic
permutations of the factors in the tensor product $\t^3\g$. Since
$$\la [\ps(a),[b,c]_{\gs}]_{\gs},x\ra=\la a\t b\t c,(\p\t\d)\d(x)\ra,$$
it is clear that $[\cdot,\cdot]_{\gs}$ satisfies the Hom-Jacobi
identity if and only if $\jac_{_{\d}}$ is the zero map.

\begin{lem}
Let \hlg\ be a weakly involutive Hom-Lie algebra. Define a linear
map $\d:\g\lr\g\otimes \g$ by Eq. (\ref{definition of d}) with
some $r\in\g\otimes \g$ satisfying $(\p\t\id)r=(\id\t\p)r$ and
$[x,r+\sigma(r)]\xg=0$ for all $x\in\g$. Then
\begin{equation}
\jac_{_{\d}}(x)=\ad_{\p(x)}[r,r]\xg:=(\ad_{\p(x)}\otimes\p\otimes\p+
\p\otimes\ad_{\p(x)}\otimes
\p+\p\otimes\p\otimes \ad_{\p(x)})[r,r]_\g
\end{equation}
for all $x\in\g$.
\label{lem:jac}
\end{lem}

\begin{proof}
Let $r=\xqh{i}\x i\t\y i$. The following computations make repeated
use of the Hom-Jacobi identity in $\g$, and $[x,r+\sigma(r)]\xg=0$
for all $x\in\g$ which is equivalent to
\begin{equation}\label{r with g-invariant symmetric part explicitly}
\xqh{i}([x,\x i]\xg\t\p(\y i)+\p(\x i)\t[x,\y i]\xg)=\xqh{i}(-[x,\y
i]\xg\t\p(\x i)-\p(\y i)\t[x,\x i]\xg).
\end{equation}

If we write out $\jac_{_{\d}}(x)$ explicitly, using the expression
$r=\xqh{i} \x i\t\y i$ and with a summation over repeated indices
understood, we have
\begin{equation}
\begin{split}
&\jac_{_{\d}}(x)\\
=&[\p(x),\p(\x i)]\xg\t[\p(\y i),\x j]\xg\t\p(\y j)+[\p(x),\p(\x i)]\xg\t\p(\x j)\t[\p(\y i),\y j]\xg\\
&+\p^2(\x i)\t[[x,\y i]\xg,\x j]\xg\t\p(\y j)+\p^2(\x i)\t\p(\x j)\t[[x,\y i]\xg,\y j]\xg\\
&+\p(\y j)\t[\p(x),\p(\x i)]\xg\t[\p(\y i),\x j]\xg+[\p(\y i),\y j]\xg\t[\p(x),\p(\x i)]\xg\t\p(\x j)\\
&+\p(\y j)\t\p^2(\x i)\t[[x,\y i]\xg,\x j]\xg+[[x,\y i]\xg,\y j]\xg\t\p^2(\x i)\t\p(\x j)\\
&+[\p(\y i),\x j]\xg\t\p(\y j)\t[\p(x),\p(\x i)]\xg+\p(\x j)\t[\p(\y i),\y j]\xg\t[\p(x),\p(\x i)]\xg\\
&+[[x,\y i]\xg,\x j]\xg\t\p(\y j)\t\p^2(\x i)+\p(\x j)\t[[x,\y i]\xg,\y j]\xg\t\p^2(\x i).\\
\end{split}
\label{Jac explicitly}
\end{equation}
Since \hlg\ is a weakly involutive Hom-Lie algebra and
$(\p\t\id)r=(\id\t\p)r$, the first term in Eq.~(\ref{Jac explicitly})
becomes
\begin{align*}
[\p(x),\p(\x i)]\xg\t[\p(\y i),\x j]\xg\t\p(\y j)=&[\p(x),\p(\x i)]\xg\t[\p(\y i),\p^2(\x j)]\xg\t\p(\y j)\\
=&[\p(x),\p(\x i)]\xg\t[\p(\y i),\p(\x j)]\xg\t\p^2(\y j).
\end{align*}
Doing the same to other terms in Eq.~(\ref{Jac explicitly}), we obtain
\begin{align*}
&\jac_{_{\d}}(x)\\
=&[\p(x),\p(\x i)]\xg\t[\p(\y i),\p(\x j)]\xg\t\p^2(\y j)+[\p(x),\p(\x i)]\xg\t\p^2(\x j)\t[\p(\y i),\p(\y j)]\xg\\
&+\p^2(\x i)\t[[x,\y i]\xg,\p(\x j)]\xg\t\p^2(\y j)+\p^2(\x i)\t\p^2(\x j)\t[[x,\y i]\xg,\p(\y j)]\xg\\
&+\p^2(\y j)\t[\p(x),\p(\x i)]\xg\t[\p(\y i),\p(\x j)]\xg+[\p(\y i),\p(\y j)]\xg\t[\p(x),\p(\x i)]\xg\t\p^2(\x j)\\
&+\p^2(\y j)\t\p^2(\x i)\t[[x,\y i]\xg,\p(\x j)]\xg+[[x,\y i]\xg,\p(\y j)]\xg\t\p^2(\x i)\t\p^2(\x j)\\
&+[\p(\y i),\p(\x j)]\xg\t\p^2(\y j)\t[\p(x),\p(\x i)]\xg+\p^2(\x j)\t[\p(\y i),\p(\y j)]\xg\t[\p(x),\p(\x i)]\xg\\
&+[[x,\y i]\xg,\p(\x j)]\xg\t\p^2(\y j)\t\p^2(\x i)+\p^2(\x j)\t[[x,\y i]\xg,\p(\y j)]\xg\t\p^2(\x i).
\end{align*}

Then we write out $\jac_{_{\d}}(x)-\ad_{\p(x)}[r,r]\xg$ explicitly. Note that $\jac_{_{\d}}(x)$ is a sum of twelve terms and that
$\ad_{\p(x)}[r,r]\xg$ is a sum of nine terms, but two terms appear in both sums and hence are canceled. Thus $\jac_{_{\d}}(x)-\ad_{\p(x)}[r,r]\xg$ is a sum of
seventeen terms. After rearranging the terms suitably, we obtain
\begin{align*}
&\jac_{_{\d}}(x)-\ad_{\p(x)}[r,r]\xg\\
=& [[x,\y i]\xg,\p(\x j)]\xg\t\p^2(\y j)\t\p^2(\x i)-[\p(x),[\x i,\x j]\xg]\xg\t\p^2(\y i)\t\p^2(\y j)\\
&+[[x,\y i]\xg,\p(\y j)]\xg\t\p^2(\x i)\t\p^2(\x j)+[\p(\y i),\p(\y j)]\xg\t[\p(x),\p(\x i)]\xg\t\p^2(\x j)\\
&-\p[\x i,\x j]\xg\t[\p(x),\p(\y i)]\xg\t\p^2(\y j)\\
&+[\p(\y i),\p(\x j)]\xg\t\p^2(\y j)\t[\p(x),\p(\x i)]\xg-\p[\x i,\x j]\xg\t\p^2(\y i)\t[\p(x),\p(\y j)]\xg\\
&+\p^2(\x i)\t[[x,\y i]\xg,\p(\x j)]\xg\t\p^2(\y j)+\p^2(\x i)\t\p^2(\x j)\t[[x,\y i]\xg,\p(\y j)]\xg\\
&+\p^2(\y j)\t[\p(x),\p(\x i)]\xg\t[\p(\y i),\p(\x j)]\xg+\p^2(\y j)\t\p^2(\x i)\t[[x,\y i]\xg,\p(\x j)]\xg\\
&+\p^2(\x j)\t[\p(\y i),\p(\y j)]\xg\t[\p(x),\p(\x i)]\xg+\p^2(\x j)\t[[x,\y i]\xg,\p(\y j)]\xg\t\p^2(\x i)\\
&-\p^2(\x i)\t[\p(x),[\y i,\x j]\xg]\xg\t\p^2(\y j)-\p^2(\x i)\t\p[\y i,\x j]\xg\t[\p(x),\p(\y i)]\xg\\
&-\p^2(\x i)\t[\p(x),\p(\x j)]\xg\t\p[\y i,\y j]\xg-\p^2(\x
i)\t\p^2(\x j)\t[\p(x),[\y i,\y j]\xg]\xg.
\end{align*}
Interchanging the indices $i$ and $j$ in the first term and using
the Hom-Jacobi identity in $\g$, the first term becomes
\begin{align*}
-[\p(x),[\x i,\y j]\xg]\xg\t\p^2(\y i)\t\p^2(\x j)-[\p(\y j),[x,\x
i]\xg]\xg\t\p^2(\y i)\t\p^2(\x j).
\end{align*}
The sum of $-[\p(\y j),[x,\x i]\xg]\xg\t\p^2(\y i)\t\p^2(\x j)$ and
the third and fourth terms is
\begin{align*}
&(\ad_{\p(\y j)}\t\p)(-[x,\x i]\xg\t\p(\y i)-[x,\y i]\xg\t\p(\x i)-\p(\y i)\t[x,\x i]\xg)\t\p^2(\x j)\\
=&(\ad_{\p(\y j)}\t\p)(\p(\x i)\t[x,\y i]\xg)\t\p^2(\x j)\\
=& -\p[\x i,\y j]\t\p[x,\y i]\xg\t\p^2(\x j),
\end{align*}
using the Eq. (\ref{r with g-invariant symmetric part explicitly}).

Similarly, the sum of $-\p[\x i,\y j]\t\p[x,\y i]\xg\t\p^2(\x j)$
and the fifth term becomes
\begin{align*}
\p^2(\y j)\t\p[x,\y i]\xg\t\p[\x i,\x j]\xg+\p^2(\x j)\t\p[x,\y
i]\xg\t\p[\x i,\y j]\xg,
\end{align*}
and the sum of the sixth and seventh terms is
\begin{align*}
\p^2(\x j)\t\p^2(\y i)\t[\p(x),[\x i,\y j]\xg]\xg+\p^2(\y
j)\t\p^2(\y i)\t[\p(x),[\x i,\x j]\xg]\xg.
\end{align*}

Finally, the sum of $-[\p(\y j),[x,\x i]\xg]\xg\t\p^2(\y
i)\t\p^2(\x j)$ and the second term in the sum of the expression of $\jac_{_{\d}}(x)-\ad_{\p(x)}[r,r]\xg$ becomes
\begin{align*}
&[\p(x),\p(\x j)]\xg\t\p^2(\y i)\t[\p(\x i),\p(\y j)]\xg+[\p(x),\p(\y j)]\xg\t\p^2(\y i)\t[\p(\x i),\p(\x j)]\xg\\
=& -\p^2(\y j)\t\p^2(\y i)\t[\p(\x i),[x,\x j]\xg]\xg-\p^2(\x
j)\t\p^2(\y i)\t[\p(\x i),[x,\y j]\xg]\xg.
\end{align*}

Inserting these results, we find that the expression of
$\jac_{_{\d}}(x)-\ad_{\p(x)}[r,r]\xg$ can be written in the form
$\xqh{i}(\p^2(\x i)\t u_i+\p^2(\y i)\t v_i)$ for certain $u_i,v_i\in\g$.
In fact,
\begin{align*}
u_i=& [\p(\y j),\p(\y i)]\xg\t[\p(x),\p(\x j)]\xg-[\p(\y i),\p(\x j)]\xg\t[\p(x),\p(\y j)]\xg\\
&+\p^2(\y j)\t[\p(x),[\x j,\y i]\xg]\xg-\p^2(\x j)\t[\p(x),[\y i,\y j]\xg]\xg\\
&+\p[x,\y j]\xg\t\p[\x j,\y i]\xg-[\p(x),\p(\x j)]\xg\t[\p(\y i),\p(\y j)]\xg-[\p(x),[\y i,\x j]\xg]\xg\t\p^2(\y j)\\
&+[[x,\y j]\xg,\p(\y i)]\xg\t\p^2(\x j)\\
&+[[x,\y i]\xg,\p(\x j)]\xg\t\p^2(\y j)+\p^2(\x j)\t[[x,\y
i]\xg,\p(\y j)]\xg+\p^2(\y j)\t[[x,\y i]\xg,\p(\x j)]\xg.
\end{align*}
On the right-hand side, the sum of the first four terms is zero by
Eq. (\ref{r with g-invariant symmetric part explicitly}), and that
of the next three terms becomes
$$[\p(x),[\y i,\y j]\xg]\xg\t\p^2(\x j).$$
By the Hom-Jacobi identity in $\g$, the sum of $[\p(x),[\y i,\y
j]\xg]\xg\t\p^2(\x j)$ and the eighth term is
$$[[x,\y i]\xg,\p(\y j)]\xg\t\p^2(\x j).$$
Finally, the sum of $[[x,\y i]\xg,\p(\y j)]\xg\t\p^2(\x j)$ and
the last three terms becomes
\begin{align*}
&[[x,\y i]\xg,\x j]\xg\t\p(\y j)+\p(\x j)\t[[x,\y i]\xg,\y j]\xg+\p(\y j)\t[[x,\y i]\xg,\x j]\xg+[[x,\y i]\xg,\y j]\xg\t\p(\x j)\\
&= 0,
\end{align*}
if we replace $x$ in Eq.~(\ref{r with g-invariant symmetric part
explicitly}) by $[x,\y i]\xg$. Hence, we get $u_i=0$. A similar, but
shorter, argument proves that
\begin{align*}
v_i=& [\p(x),\p(\x j)]\xg\t[\p(\y j),\p(\x i)]\xg+\p[x,\y j]\xg\t\p[\x j,\x i]\xg+\p^2(\x j)\t[[x,\y
j]\xg,\p(\x i)]\xg\\
&+\p^2(\y j)\t[[x,\x i]\xg,\p(\x j)]\xg+\p^2(\y j)\t[\p(x),[\x j,\x i]\xg]\xg\\
=&0.
\end{align*}
Hence the conclusion holds.
\end{proof}

\begin{thm}\label{main theorem}
    Let \hlg\ be a weakly involutive Hom-Lie algebra. Define a bilinear map $[\cdot,\cdot]_{\gs}:\gs\otimes \gs\lr\gs$ by Eq.~(\ref{dual bracket}), where
     $\d$ is defined by Eq.~(\ref{definition of d}) with some $r\in\g\otimes \g$ satisfying $(\p\t\id)r=(\id\t\p)r$. Then \hlgs\ is a weakly involutive Hom-Lie algebra if and only if the following conditions are satisfied:
    \begin{itemize}
        \item[(i)]  $[x,r+\sigma(r)]\xg=0$ for all $x\in\g$,
        \item[(ii)]  $\ad_{\p(x)}[r,r]\xg=0$ for all $x\in\g$.
    \end{itemize}
    Under these conditions, $(\g,\gs)$ is a coboundary Hom-Lie bialgebra.
\end{thm}

\begin{proof}
    The bracket $[\cdot,\cdot]_{\gs}$ is determined by the cobracket $\d(x)=[x,r]\xg$. Note that $(\p\t\id)r=(\id\t\p)r$. Applying Lemma~\ref{lem:delta}, we find that $\ps$ is an algebra
    homomorphism with respect to $[\cdot,\cdot]_{\gs}$ and $[{\ps}^2(a),b]_{\gs}=[a,b]_{\gs}$. Hence \hlgs\ is a weakly involutive Hom-Lie algebra if and only if $[\cdot,\cdot]_{\gs}$ is skew-symmetric and satisfies the Hom-Jacobi identity.

    The proof that $r$ satisfies (i) if and only if $[\cdot,\cdot]_{\gs}$ is skew-symmetric is straightforward. The proof that (ii) is equivalent to the condition that $[\cdot,\cdot]_{\gs}$ satisfies the Hom-Jacobi identity follows from Lemma~\ref{lem:jac}.

    Since $\d(x)=[x,r]\xg$ and $(\p\t\id)r=(\id\t\p)r$, by Lemma \ref{lem:delta}.\ref{it:comp}, the compatibility conditions for a Hom-Lie bialgebra in Definition \ref{defi Hom-Lie bialgebra} hold naturally. Therefore the conclusion follows.
\end{proof}

\begin{rem}
An easy way to satisfy conditions (i) and $(\i\i)$ in
Theorem \ref{main theorem} is to assume that $r$ is skew-symmetric and
\begin{equation}\label{Hom-Yang-Baxter equation}
[r,r]\xg=0
\end{equation}
respectively. Eq.~(\ref{Hom-Yang-Baxter equation}) is the {\bf classical Hom-Yang-Baxter equation}. A {\bf
quasitriangular Hom-Lie bialgebra} is a coboundary Hom-Lie
bialgebra, in which $r$ is a solution of the classical
Hom-Yang-Baxter equation. A {\bf triangular Hom-Lie bialgebra} is a
coboundary Hom-Lie bialgebra, in which $r$ is a skew-symmetric
solution of the classical Hom-Yang-Baxter equation.
\end{rem}

\begin{pro}
    Let $(\g,\gs)$ be a Hom-Lie bialgebra. Then there is a canonical Hom-Lie bialgebra structure on the direct sum $\g\oplus \gs$ of the underlying vector spaces of $\g$ and $\gs$ such that $\p:\g\hookrightarrow \g\oplus \gs$ and $\p^*:\gs\hookrightarrow \g\oplus \gs$ into the two summands are homomorphisms of Hom-Lie bialgebras. Here the Hom-Lie bialgebra structure on $\gs$ is $(\gs,-\d_{\g^*})$, where $\d_{\g^*}$ is given by Eq.~(\ref{eq:dual}).
\end{pro}

\begin{proof}
    Let $r\in \g\t \gs\subset (\g\oplus \gs)\t(\g\oplus \gs)$ correspond to the identity map $\mathrm{id}:\g\longrightarrow \g$. Let $\{e_1,\ldots,e_n\}$ be a basis of $\g$ and $\{f_1,\ldots,f_n\}$ be its dual basis. Then $r=\sum\limits_{i} e_i\t f_i$. We denote $\g\oplus\gs$ by $\mathcal{HD}(\g)$, and $\p\oplus\ps$ by $\phi_{_\mathcal{HD}}$. By Corollary \ref{bialgebra, weakly}, there is a weakly involutive Hom-Lie algebra structure $(\mathcal{HD}(\g),[\cdot,\cdot]_{_{\mathcal{HD}(\g)}},\phi_{_\mathcal{HD}})$ on $\mathcal{HD}(\g)$.  Since
    \begin{equation*}
    \la \sum\limits_{i}\p(e_i)\t f_i,f_s\t e_t\ra=\la \p(e_t) ,f_s\ra=\la e_t,\p^*(f_s)\ra,
    \end{equation*}
    \begin{equation*}
    \la \sum\limits_{i}e_i\t \p^*(f_i),f_s\t e_t\ra=\la e_t,\p^*(f_s)\ra,
    \end{equation*}
    we have $\sum\limits_{i}\p(e_i)\t f_i=\sum\limits_{i}e_i\t \p^*(f_i)$. Note that $\phi_{_\mathcal{HD}}(x+a)=\p(x)+\p^*(a)$ for all $x\in \g$, $a\in \gs$. We
    get $(\phi_{_\mathcal{HD}}\t \mathrm{id})r=(\mathrm{id}\t\phi_{_\mathcal{HD}})r$. Set
    $$  \d_{\mathcal{HD}}(u)=[u,r]_{_{\mathcal{HD}(\g)}},\ \forall\  u\in\mathcal{HD}(\g).$$

    Since
    \begin{align*}
    & \la [r,r]_{_{\mathcal{HD}(\g)}},(e_s+f_t)\t(e_k+f_l)\t(e_p+f_q) \ra\\
    =& \xqh{i,j} \la  [e_i,e_j]_{_{\mathcal{HD}(\g)}}\t\phi_{_{\mathcal{HD}}}(f_i)\t\phi_{_{\mathcal{HD}}}(f_j)+\phi_{_{\mathcal{HD}}}(e_i)\t[f_i,e_j]_{_{\mathcal{HD}(\g)}}\t\phi_{_{\mathcal{HD}}}(f_j)\\
    & +\phi_{_{\mathcal{HD}}}(e_i)\t\phi_{_{\mathcal{HD}}}(e_j)\t[f_i,f_j]_{_{\mathcal{HD}(\g)}},(e_s+f_t)\t(e_k+f_l)\t(e_p+f_q) \ra\\
    =& \xqh{i,j} \la [e_i,e_j]\xg\t\p^*(f_i)\t\p^*(f_j)+\p(e_i)\t (\dao_{f_i} (e_j)-\ado_{e_j} (f_i) ) \t\p^*(f_j)\\
    & +\p(e_i)\t\p(e_j)\t[f_i,f_j]_{\gs},(e_s+f_t)\t(e_k+f_l)\t(e_p+f_q) \ra\\
    =& \xqh{i,j} (\la [e_i,e_j]\xg,f_t\ra \la \p^*(f_i),e_k\ra \la \p^*(f_j),e_p\ra+\la \p(e_i),f_t\ra \la e_j,-[\p^*f_i,f_l]_{\gs}\ra \la\p^*(f_j),e_p\ra\\
    &-\la \p(e_i),f_t\ra \la f_i,-[\p(e_j),e_k]\xg\ra \la\p^*(f_j),e_p\ra+\la \p(e_i),f_t\ra \la\p(e_j),f_l\ra \la [f_i,f_j]_{\gs},e_p\ra)\\
    =& \la [\p (e_k),\p(e_p)]\xg,f_t\ra -\la \p(e_p),[(\p^*)^2 f_t,f_l]_{\gs}\ra\\
    &+\la \p^*(f_t),[\p^2 (e_p),e_k]\xg\ra+\la [\p^*(f_t),\p^*(f_l)]_{\gs},e_p\ra
    =0,
    \end{align*}
    we get $[r,r]_{_{\mathcal{HD}(\g)}}=0$. Similarly, we prove that $[u,r+\sigma(r)]_{_{\mathcal{HD}(\g)}}=0$ for all $u\in\mathcal{HD}(\g)$.  Hence $\mathcal{HD}(\g)$ is a quasitriangular Hom-Lie bialgebra by Theorem \ref{main theorem}.

    For $e_i\in \g$, we have
    \begin{align*}
    & \la \d_{\mathcal{HD}}(\p(e_i)),(e_s+f_t)\t(e_k+f_l)\ra  \\
    =& \sum\limits_{j} \la [\p(e_i),e_j]_{_{\mathcal{HD}(\g)}}\t \phi_{_{\mathcal{HD}}}(f_j)+ \phi_{_{\mathcal{HD}}}(e_j)\t [\p(e_i),f_j]_{_{\mathcal{HD}(\g)}},(e_s+f_t)\t(e_k+f_l)\ra \\
    =& \sum\limits_{j} (\la [\p(e_i),e_j]\xg,f_t\ra \la\p^*(f_j),e_k\ra+ \la \p(e_j),f_t\ra \la f_j,-[\p^2(e_i),e_k]\xg\ra\\
    & -\la\p(e_j),f_t\ra \la \p(e_i),-[\p^*(f_j),f_l]_{\gs}\ra)\\
    =& \la [\p(e_i),\p(e_k)]\xg,f_t\ra-\la \p^*(f_t),[e_i,e_k]\xg\ra+\la \p(e_i),[(\p^*)^2 (f_t),f_l]_{\gs}\ra  \\
    =& \la e_i,[\p^*(f_t),\p^*(f_l)]_{\gs}\ra\\
    =& \la (\p\t\p)\d(e_i),(e_s+f_t)\t(e_k+f_l)\ra.
    \end{align*}
    So $\d_{\mathcal{HD}}(\p(e_i))=(\p\t\p)\d(e_i)$. Therefore $\p:\g\hookrightarrow \g\oplus \gs$ is a homomorphism of Hom-Lie bialgebras. Similarly, $\p^*:\gs\hookrightarrow \g\oplus \gs$ is also a homomorphism of Hom-Lie bialgebras since $\d_{\mathcal{HD}}(\p^*(f_i))=(\p^*\t\p^*)(-\d_{\gs})(f_i)$.
\end{proof}

\begin{rem}
With the above Hom-Lie bialgebra structure, $\g\oplus \gs$ is called the {\bf Hom-double} of $\g$, and is denoted by $\mathcal{HD}(\g)$.
\end{rem}

For any $r\in\g\otimes \g$, the induced linear map $\rl:\gs\lr\g$
is defined by
$$\la \rl(a),b\ra=\la r,a\t b\ra.$$
Then it is easy to see that $(\p\t\id)r=(\id\t\p)r$ is equivalent to
$$\p\, \rl=\rl\, \ps.$$
Setting $r_{21}=\sigma(r)$, we have

\begin{pro}\label{pp:dual bracket}
    Let \hlg\ be a weakly involutive Hom-Lie algebra and $r\in\g\otimes \g$ such that $\p\, \rl=\rl\, \ps$. Define linear maps $\d:\g\lr\g\otimes \g$ by Eq.~(\ref{definition of d}) and $[\cdot,\cdot]_{\gs}:\gs\otimes \gs\lr\gs$ by Eq.~(\ref{dual bracket}). Then we have
    \begin{equation}
    [a,b]_{\gs}=\ado_{\rl(a)} b+\ado_{\srl(b)} a.
    \end{equation}
\end{pro}

\begin{proof}
    For convenience, we assume that $r=x\t y$ is a pure tensor. Then we have
    \begin{align*}
    & \la z,[a,b]_{\gs}\ra \\
    =& \la \d(z),a\t b\ra=\la [z,r]\xg,a\t b\ra=\la [z,x]\xg\t\p(y)+\p(x)\t[z,y]\xg,a\t b\ra\\
    =& \la [z,x]\xg,a\ra\la\p(y),b\ra+\la\p(x),a\ra\la [z,y]\xg,b\ra
    = \la [z,\la y,\ps(b)\ra x]\xg,a\ra+\la [z,\la x,\ps(a)\ra y]\xg,b\ra\\
    =& \la [z,\srl\, \ps(b)]\xg,a\ra+\la [z,\rl\, \ps(a)]\xg,b\ra
    = \la [z,\p\,  \srl(b)]\xg,a\ra+\la [z,\p\,  \rl(a)]\xg,b\ra\\
    =& \la z,\ado_{\rl(a)} b+\ado_{\srl(b)} a\ra.
    \end{align*}
    Hence the conclusion holds.
\end{proof}

\begin{lem}\label{lem 1}
    Let \hlg\ be a weakly involutive Hom-Lie algebra. If $r\in\g\otimes \g$ satisfies $\p\, \rl=\rl\, \ps$ and $[\cdot,\cdot]_{\gs}:\gs\otimes \gs\lr\gs$ is given by Eq. (\ref{dual bracket}), where
     $\d$ is defined by Eq.~(\ref{definition of d}), then we have
    \begin{equation}
    [\rl\, \ps(a),\rl\, \ps(b)]\xg-\rl\, \ps[a,b]_{\gs}=[r,r]\xg(a,b).
    \end{equation}
\end{lem}

\begin{proof}
    For any $a,b,c\in\g$, we have
    \begin{align*}
    \la \rl\, \ps[a,b]_{\gs},c\ra=& \la \ado_{\rl(a)} b+\ado_{\srl(b)} a,\p\, \srl(c)\ra\\
    =& -\la b,[\p\, \rl(a),\p\, \srl(c)]\xg\ra-\la a,[\p\, \srl(b),\p\, \srl(c)]\xg\ra\\
    =& -\la b,[\rl\, \ps(a),\srl\, \ps(c)]\xg\ra-\la a,[\srl\, \ps(b),\srl\, \ps(c)]\xg\ra.
    \end{align*}
    It is straightforward to deduce that
    \begin{align*}
    & [r,r]\xg(a,b,c) \\
    =& \la a,[\srl\, \ps(b),\srl\, \ps(c)]\xg\ra+\la b,[\rl\, \ps(a),\srl\, \ps(c)]\xg\ra+\la c,[\rl\, \ps(a),\rl\, \ps(b)]\xg\ra.
    \end{align*}
Therefore the conclusion holds.
\end{proof}

\begin{cor}\label{cor homomorphism}
Let $(\g,\d)$ be a quasitriangular Hom-Lie bialgebra, where $\d$ is
given by Eq.~(\ref{definition of d}) for a solution $r$ of the classical
Hom-Yang-Baxter equation. Then $\rl\,  \ps:\gs\lr\g$ is a
homomorphism of Hom-Lie algebras.
\end{cor}

Let \hlg\ be a weakly involutive Hom-Lie algebra and $r\in\g\otimes
\g$ be invertible (that is, $\rl$ is invertible). Define a bilinear form
$B\in(\g\otimes \g)^*$  by
\begin{equation}
B(x,y)=\la (\rl)^{-1}(x),y\ra.
\end{equation}
Then it is easy to see that $B$ is skew-symmetric if and only if $r$
is skew-symmetric.

\begin{pro}
    With the above notations, suppose that  $r$ is skew-symmetric. Then $\p\, \rl=\rl\, \ps$ and $r$ is a solution of
     the classical Hom-Yang-Baxter equation (\ref{Hom-Yang-Baxter equation}) if and
    only if
    \begin{equation}
    B(\p(x),[y,z]\xg)+B(\p(y),[z,x]\xg)+B(\p(z),[x,y]\xg)=0, B(\p(x),y)=B(x,\p(y)).
    \end{equation}
\end{pro}

\begin{proof}
    For any $x,y,z\in\g$, set $x=\rl(a)$, $y=\rl(b)$, $z=\rl(c)$. Since $\rl\, \ps=\p\, \rl$, we have
    \begin{equation*}
    \begin{split}
    B(\p(x),y)=\la (\rl)^{-1}\, \p(x),y\ra=\la \ps\, (\rl)^{-1}(x),y\ra=\la (\rl)^{-1}(x),\p(y)\ra=B(x,\p(y)).
    \end{split}
    \end{equation*}
    By Lemma \ref{lem 1}, if $r$ satisfies the classical Hom-Yang-Baxter equation, then we have
    \begin{align*}
    B(\p(x),[y,z]\xg)=& B(x,\p [y,z]\xg)=\la a,[\p\, \rl(b),\p\, \rl(c)]\xg\ra=\la a,\rl\, \ps[b,c]_{\gs}\ra\\
    =& \la -\p\, \rl(a),\ado_{\rl(b)} c-\ado_{\rl(c)} b\ra\\
    =& \la c,[\p\, \rl(b),\p\, \rl(a)]\xg\ra-\la b,[\p\, \rl(c),\p\, \rl(a)]\xg\ra\\
    =& B(z,\p [y,x]\xg)-B(y,\p [z,x]\xg)\\
    =& -B(\p(z),[x,y]\xg)-B(\p(y),[z,x]\xg).
    \end{align*}
    The converse can be proved similarly.
\end{proof}

\section{Operator forms of the classical Hom-Yang-Baxter equation}
\label{sec:oop}
In this section, we give a further interpretation of the classical
Hom-Yang-Baxter equation. We discuss the relationship between an
\o-operator associated to an arbitrary weakly involutive representation and the classical Hom-Yang-Baxter equation, which
leads to a construction of solutions of the classical
Hom-Yang-Baxter equation by means of \o-operators and
Hom-left-symmetric algebras.

\begin{defi}[\cite{BS}]
Let \hlg\ be a Hom-Lie algebra and \rep\ be a representation of \hlg. A
linear map $T:\v\lr\g$ is called an {\bf \o-operator} associated to \rep\
if $T$ satisfies
\begin{itemize}
    \item[(i)]  $T\, \b=\p\,  T$,
    \item[(ii)]  $[T(u),T(v)]\xg=T(\r(T(u))v-\r(T(v))u)$, $\forall\ u,v\in\v$.
\end{itemize}
\end{defi}

\begin{exa}
Let \hlg\ be a weakly involutive Hom-Lie algebra. Suppose that
$r\in\j^2\g$ satisfies Eq. (\ref{ra*=ar}). By Lemma \ref{lem 1},
$r$ satisfies the classical Hom-Yang-Baxter equation
(\ref{Hom-Yang-Baxter equation}) if $\rl$ is an \o-operator
associated to the representation $(\gs,\ps,\ado)$. Conversely, if
$r$ satisfies the classical Hom-Yang-Baxter equation
(\ref{Hom-Yang-Baxter equation}) and in addition, $\p$ is
invertible,  then $\rl$ is an \o-operator associated to the
representation $(\gs,\ps,\ado)$.
\end{exa}

There is a class of Hom-Lie algebras coming from the
following structure:
\begin{defi}[\cite{MS2,DY}]
A {\bf Hom-left-symmetric algebra} is a triple $(\v,\cdot,\psi)$
consisting of a linear space \v, a bilinear map $\cdot:\v\t\v\lr\v$
and an algebra homomorphism $\psi:\v\lr\v$ satisfying
\begin{equation}
(u\cdot v)\cdot\psi(w)-\psi(u)\cdot(v\cdot w)=(v\cdot
u)\cdot\psi(w)-\psi(v)\cdot(u\cdot w),\ \forall\ u,v,w\in\v.
\end{equation}
\end{defi}

Indeed,
\begin{pro}[\cite{BS}]\label{commutator Hom-Lie algebra}
Let \hls\ be a Hom-left-symmetric algebra. Then
\begin{itemize}
    \item[(i)]  $(\g(\v),[\cdot,\cdot]_{\v},\psi)$ is a Hom-Lie algebra, where $\g(\v)=\v$ as a vector space, and the bracket $[\cdot,\cdot]_{\v}$ is given by
    \begin{equation*}
    [u,v]_{\v}=u\cdot v-v\cdot u.
    \end{equation*}
    We call $(\g(\v),[\cdot,\cdot]_{\v},\psi)$ the {\bf commutator Hom-Lie algebra}.
        \item[(ii)]  Let $L:\v\lr\gl(\v)$ be the linear map with $u\mapsto L_u$, where $L_u(v)=u\cdot v$ for any $u,v\in\v$. Then $(\v,\psi,L)$ is a representation of the Hom-Lie algebra $(\g(\v),[\cdot,\cdot]_{\v},\psi)$ on \v.
\end{itemize}
\end{pro}

\begin{exa}
Let \hls\ be a Hom-left-symmetric algebra. Then $\id$ is an
\o-operator of the commutator Hom-Lie algebra $(\g(\v),[\cdot,\cdot]_{\v},\psi)$
associated to the representation $(\v,\psi,L)$.
\end{exa}

\begin{cor}\label{2 o-operator}
Let \hls\ be a Hom-left-symmetric algebra. Then $(\v,\psi,L)$ is a
weakly involutive representation of the commutator Hom-Lie algebra
$(\g(\v),[\cdot,\cdot]_{\v},\psi)$ if and only if
$$u\cdot v=\psi^2(u)\cdot v,\ \forall\ u,v\in\v.$$
Under this condition, $\psi^2$ is an \o-operator of
$(\g(\v),[\cdot,\cdot]_{\v},\psi)$ associated to $(\v,\psi,L)$.
\end{cor}

\begin{proof}
By definition, $(\v,\psi,L)$ is a weakly involutive representation
of $(\g(\v),[\cdot,\cdot]_{\v},\psi)$ if and only if
$L_{u}=L_{\psi^2(u)}$ for all $u\in\v$, i.e. $u\cdot v=\psi^2(u)\cdot
v$ for all $u,v\in\v$. Under this condition, we have
\begin{align*}
\psi^2(L_{\psi^2(u)}(v)-L_{\psi^2(v)}(u))=& \psi^2(u\cdot v-v\cdot
u)=\psi^2(u)\cdot \psi^2(v)-\psi^2(v)\cdot
\psi^2(u)\\
=&[\psi^2(u),\psi^2(v)]_{\v},
\end{align*}
and $\psi^2\,  \psi=\psi\,  \psi^2$. Hence the conclusion holds.
\end{proof}

Let \hlg\ be a Hom-Lie algebra and \rep\ be a weakly involutive
representation. Then we have the semidirect product Hom-Lie algebra
$\g\ltimes_{\ro}V^*$, and any linear map $T:\v\lr\g$ can be viewed
as an element $\overline{T}\in (\g\oplus V^*)\t (\g\oplus V^*)$ via
\begin{equation}
\overline{T}(a+u,b+v)=\la T(u),b\ra,\ \forall\
a+u,b+v\in\gs\oplus\v.
\end{equation}
Then $r=\overline{T}-\sigma(\overline{T})$ is skew-symmetric.

    \begin{lem}\label{lem 2}
With the above notations, let $T:\v\lr\g$ be a linear map satisfying $T\, \b=\p\,  T$. Then $r=\overline{T}-\sigma(\overline{T})$ satisfying $(\phi_{_{d}}\t \mathrm{Id})r=(\mathrm{Id}\t \phi_{_{d}})r$, where $\phi_{_{d}}=\p\oplus \b^*$.
    \end{lem}
\begin{proof}
Let $\{v_1,\ldots,v_n\}$ be a basis of \v\ and $\{v^1,\ldots,v^n\}$
be its dual basis. Use the Einstein summation convention,
$r$ can be expressed by $r=v^i\t T(v_i)-T(v_i)\t v^i$.
Then we have
\begin{align*}
(\phi_{_{d}}\t \mathrm{Id})r=& \b^*(v^i)\t T(v_i)-\p\,  T(v_i)\t v^i,\\
(\mathrm{Id}\t \phi_{_{d}})r=& v^i\t \p\,  T(v_i)-T(v_i)\t \b^*(v^i).
\end{align*}
Since
\begin{align*}
\la \b^*(v^i)\t T(v_i),v_j\t a\ra=& \la \b^*(v^i),v_j\ra \la T(v_i),a\ra=\la T(\la v^i,\b(v_j)\ra v_i),a\ra=\la T(\b(v_j)),a\ra\\
=& \la v^i\t T\,  \b(v_i),v_j\t a\ra,
\end{align*}
we get $\b^*(v^i)\t T(v_i)=v^i\t T\,  \b(v_i)$. Similarly, we have $T(v_i)\t \b^*(v^i)=T\,  \b(v_i)\t v^i$. Hence the conclusion holds.
\end{proof}

\begin{thm}\label{theorem 2}
    Let \hlg\ be a Hom-Lie algebra and \rep\ be a weakly involutive representation. Let $T:\v\lr\g$ be a linear map satisfying $T\, \b=\p\,  T$. Then $r=\overline{T}-\sigma(\overline{T})$ is a solution of the classical Hom-Yang-Baxter equation in the Hom-Lie algebra $\g\ltimes_{\ro}V^*$ if $T$ is an \o-operator associated to \rep.
\end{thm}

\begin{proof}
Let $\{v_1,\ldots,v_n\}$ be a basis of \v\ and $\{v^1,\ldots,v^n\}$
be its dual basis. Use the Einstein summation convention,
$\overline{T}$ can be expressed by $\overline{T}=v^i\t T(v_i)$.
Then we have
\begin{align*}
&[r,r]_{\g\ltimes_{\ro}V^*}\\
=& -[v^i,T(v_j)]_{\g\ltimes_{\ro}V^*}\t\p\,  T(v_i)\t\b^*v^j-[T(v_i),v^j]_{\g\ltimes_{\ro}V^*}\t\b^*v^i\t\p\,  T(v_j)\\
& +[T(v_i),T(v_j)]\xg\t\b^*v^i\t\b^*v^j\\
&-\b^*v^i\t[T(v_i),T(v_j)]\xg\t\b^*v^j+\b^*v^i\t[T(v_i),v^j]_{\g\ltimes_{\ro}V^*}\t\p\,  T(v_j)\\
&+\p\,  T(v_i)\t[v^i,T(v_j)]_{\g\ltimes_{\ro}V^*}\t\b^*v^j\\
&+\b^*v^i\t\b^*v^j\t[T(v_i),T(v_j)]\xg-\b^*v^i\t\p\,  T(v_j)\t[T(v_i),v^j]_{\g\ltimes_{\ro}V^*}\\
&-\p\,  T(v_i)\t\b^*v^j\t[v^i,T(v_j)]_{\g\ltimes_{\ro}V^*}
\end{align*}
Set
\begin{equation*}
OT(u,v)=[T(u),T(v)]\xg-T(\r(T(u))v-\r(T(v))u),\ \forall\ u,v\in\v.
\end{equation*}
Note that
\begin{align*}
&[v^i,T(v_j)]_{\g\ltimes_{\ro}V^*}\t\p\,  T(v_i)\t\b^*v^j\\
=& -\la \ro(T(v_j))v^i,v_m\ra v^m\t T\, \b(v_i)\t\la \b^*v^j,v_k\ra v^k\\
=& v^m\t\la v^i,\r(\p\,  T(v_j))v_m\ra\la v^j,\b v_k\ra T\, \b(v_i)\t v^k
=v^m\t T\, \b(\r(T\,  \b^2(v_k))v_m)\t v^k\\
=& v^m\t \p\,  T(\r(\p^2\,  T(v_k))v_m)\t v^k
= v^m\t \p\,  T(\r(T(v_k))v_m)\t v^k
\end{align*}
and
\begin{align*}
&[T(v_i),T(v_j)]\xg\t\b^*v^i\t\b^*v^j\\
=& [T(v_i),T(v_j)]\xg\t\la \b^*v^i,v_m\ra v^m\t\la \b^*v^j,v_k\ra v^k\\
=&  \la v^i,\b(v_m)\ra\la v^j,\b(v_k)\ra [T(v_i),T(v_j)]\xg\t v^m\t v^k
=[T\, \b(v_m),T\, \b(v_k)]\xg\t v^m\t v^k\\
=& [\p\,  T(v_m),\p\,  T(v_k)]\xg\t v^m\t v^k
=\p \,  ([T(v_m),T(v_k)]\xg)\t v^m\t v^k.
\end{align*}
Then we get
\begin{align*}
&[r,r]_{\g\ltimes_{\ro}V^*}\\
=&\p(OT(v_i,v_j))\t v^i\t v^j-v^i\t \p(OT(v_i,v_j))\t v^j+ v^i\t v^j\t \p(OT(v_i,v_j)).
\end{align*}
Therefore $r=\overline{T}-\sigma(\overline{T})$ is a solution of
the classical Hom-Yang-Baxter equation in the Hom-Lie algebra
$\g\ltimes_{\ro}V^*$ if $T$ is an \o-operator associated to \rep.
\end{proof}

\begin{rem}
Suppose that in addition, $\p$ in Theorem \ref{theorem 2} is
invertible. Then $r=\overline{T}-\sigma(\overline{T})$ is a
solution of the classical Hom-Yang-Baxter equation in the Hom-Lie
algebra $\g\ltimes_{\ro}V^*$ if and only if $T$ is an \o-operator
associated to \rep.
\end{rem}

Combining Corollary~\ref{co:semidirect}, Theorem~\ref{main theorem}, Lemma~\ref{lem 2} and Theorem~\ref{theorem 2}, we get the following conclusion.
\begin{cor}\label{cor, coboundary Hom-Lie bi, r}
 Let \hlg\ be a weakly involutive Hom-Lie algebra and \rep\ be a weakly involutive representation satisfying $\rho(\phi_\g(x))\,  \b^2=\rho(\phi_\g(x))$ for all $x\in\g$. Let $T:\v\lr\g$ be an \o-operator associated to \rep. Then there is a coboundary Hom-Lie bialgebra structure on $\g\ltimes_{\ro}V^*$ induced by $r=\overline{T}-\sigma(\overline{T})$.
\end{cor}

\begin{cor}\label{solutions in hls}
Let \hls\ be a Hom-left-symmetric algebra. Suppose that
$$u\cdot v=\psi^2(u)\cdot v,\ \forall\ u,v\in\v.$$
Let $\{v_1,\ldots,v_n\}$ be a
basis of \v\ and $\{v^1,\ldots,v^n\}$ be its dual basis.
Then we have the following conclusions.
\begin{enumerate}
\item $r_1=v^i\j v_i$ and $r_2=v^i\j \psi^2(v_i)$ are solutions of
the classical Hom-Yang-Baxter equation in the Hom-Lie algebra
$\g(V)\ltimes_{L^\circ}V^*$, where $L^\circ:=L^*\, \phi_\g$ with the notation in Eq.~\eqref{eq:ro}.
\item Suppose in addition that the commutator Hom-Lie algebra $(\g(\v),[\cdot,\cdot]_{\v},\psi)$ is weakly involutive.
Then there are coboundary Hom-Lie bialgebra structures on $\g(V)\ltimes_{L^\circ}V^*$ induced by $r_1=v^i\j v_i$ and $r_2=v^i\j \psi^2(v_i)$ respectively.
Moreover, these two induced coboundary Hom-Lie bialgebra structures are the same.
\end{enumerate}
\end{cor}

\begin{proof}
(a) By Corollary \ref{2
o-operator}, $(\v,\psi,L)$ is a weakly involutive representation of the Hom-Lie
algebra $(\g(\v),[\cdot,\cdot]_{\v},\psi)$. Moreover, both the identity map $\id$ and $\psi^2$ are
\o-operators of $\g(\v)$ associated to $(\v,\psi,L)$. Therefore,
the conclusion follows from Theorem \ref{theorem 2}.

(b) In fact, $(\g(\v),[\cdot,\cdot]_{\v},\psi)$ is weakly involutive if and only if
$$\psi^2(u)\cdot v-v\cdot \psi^2(u)=u\cdot v-v\cdot u,\ \forall\ u,v\in\v.$$
Hence we get $v\cdot \psi^2(u)=v\cdot u$ for all $u,v\in\v$, i.e., $L_{v}\,  \psi^2=L_{v}$ for all $v\in\v$. Note that the identity map $\id$ and $\psi^2$ are \o-operators of $\g(\v)$ associated to $(\v,\psi,L)$. By Corollary~\ref{cor, coboundary Hom-Lie bi, r},
there are coboundary Hom-Lie bialgebra structures on $\g(V)\ltimes_{L^\circ}V^*$ induced by $r_1$ and $r_2$ respectively.
Define linear maps $\d_k$ $(k=1,2)$ by
\begin{equation*}
\d_k:\g(V)\ltimes_{L^\circ}V^*\lr(\g(V)\ltimes_{L^\circ}V^*)\otimes (\g(V)\ltimes_{L^\circ}V^*),
\end{equation*}
$$\d_k(a)=[a,r_k]_{_{\g(V)\ltimes_{L^\circ}V^*}},\ \forall a\in \g(V)\ltimes_{L^\circ}V^*.$$
For any $u,v\in V$ and $\xi,\e \in V^*$, we have
\begin{align*}
\d_1(u+\xi)=& [u+\xi,r_1]_{_{\g(V)\ltimes_{L^\circ}V^*}}\\
=&[u+\xi,v^i]_{_{\g(V)\ltimes_{L^\circ}V^*}}\t \psi(v_i)+\psi^*(v^i)\t [u+\xi,v_i]_{_{\g(V)\ltimes_{L^\circ}V^*}}\\
&-[u+\xi,v_i]_{_{\g(V)\ltimes_{L^\circ}V^*}}\t \psi^*(v^i)-\psi(v_i)\t [u+\xi,v^i]_{_{\g(V)\ltimes_{L^\circ}V^*}}\\
=& L^\circ_u (v^i)\t \psi(v_i)+ \psi^*(v^i)\t ([u,v_i]_{V}- L^\circ_{v_i} (\xi))-([u,v_i]_{V}- L^\circ_{v_i} (\xi))\t \psi^*(v^i)\\
&-\psi(v_i)\t L^\circ_u (v^i),\\
\d_2(u+\xi)=& [u+\xi,r_2]_{_{\g(V)\ltimes_{L^\circ}V^*}}\\
=&[u+\xi,v^i]_{_{\g(V)\ltimes_{L^\circ}V^*}}\t \psi^{3} (v_i)+\psi^*(v^i)\t [u+\xi,\psi^{2} (v_i)]_{_{\g(V)\ltimes_{L^\circ}V^*}}\\
&-[u+\xi,\psi^{2} (v_i)]_{_{\g(V)\ltimes_{L^\circ}V^*}}\t \psi^*(v^i)-\psi^{3} (v_i)\t [u+\xi,v_i]_{_{\g(V)\ltimes_{L^\circ}V^*}}\\
=& L^\circ_u (v^i)\t \psi^{3} (v_i)+ \psi^*(v^i)\t ([u,\psi^{2} (v_i)]_{V}- L^\circ_{\psi^{2} (v_i)} (\xi))\\
&-([u,\psi^{2} (v_i)]_{V}- L^\circ_{\psi^{2} (v_i)} (\xi))\t \psi^*(v^i)-\psi^{3} (v_i)\t L^\circ_u (v^i)\\
=& L^\circ_u (v^i)\t \psi^{3} (v_i)+ \psi^*(v^i)\t ([u,v_i]_{V}- L^\circ_{v_i} (\xi))\\
&-([u,v_i]_{V}- L^\circ_{v_i} (\xi))\t \psi^*(v^i)-\psi^{3} (v_i)\t L^\circ_u (v^i).
\end{align*}
Since
\begin{align*}
\la L^\circ_u \,  (\psi^*)^{2}(\xi),v\ra=& -\la \xi,\psi^{2}(\psi(u)\cdot v)\ra=-\la \xi,\psi(u)\cdot v\ra=\la L^\circ_u(\xi),v\ra,
\end{align*}
we get $L^\circ_u \,  (\psi^*)^{2}=L^\circ_u$. Then we have
\begin{align*}
\la L^\circ_u (v^i)\t \psi^{3} (v_i),v\t \e\ra=& \la L^\circ_u (v^i),v\ra \la \psi^{3} (v_i),\e\ra=\la L^\circ_u (\la  v_i,(\psi^*)^{3}(\e)\ra v^i),v\ra\\
=&\la L^\circ_u ((\psi^*)^{3}(\e)),v\ra=\la L^\circ_u (\psi^*(\e)),v\ra=\la L^\circ_u (v^i)\t \psi (v_i),v\t \e\ra.
\end{align*}
Hence we get $L^\circ_u (v^i)\t \psi^{3} (v_i)=L^\circ_u (v^i)\t \psi (v_i)$. Similarly, we have $\psi^{3} (v_i)\t L^\circ_u (v^i)=\psi(v_i)\t L^\circ_u (v^i)$. Thus we get $\d_1=\d_2$. Hence the conclusion holds.
\end{proof}

\smallskip

\noindent {\bf Acknowledgements} The authors acknowledge supports
from the National Natural Science Foundation of China (Grant Nos. 11425104
and 11771190) and the Fundamental Research Funds for the Central
Universities.


\begin{thebibliography}{99}
\bibitem{AMM}
F.~Ammar, S.~Mabrouk, and A.~Makhlouf, Representations and
cohomology of n-ary
  multiplicative Hom-Nambu-Lie algebras, J. Geom. Phys. \textbf{61} (2011)
  1898.

\bibitem{BM}
S.~Benayadi and A.~Makhlouf, Hom-Lie algebras with symmetric
invariant
  nondegenerate bilinear forms, J. Geom. Phys. \textbf{76} (2014) 38.

\bibitem{BEM}
M.~Bordemann, O.~Elchinger, and A.~Makhlouf, Twisting Poisson
algebras, coPoisson algebras and quantization, Trav. Math. \textbf{20} (2012) 83.

\bibitem{CWZ}
Y.~Chen, Y.~Wang, and L.~Zhang, The construction of Hom-Lie
bialgebras, J. Lie Theory, \textbf{20} (2010) 767.

\bibitem{CS-purely}
L.~Cai and Y.~Sheng, Purely Hom-Lie bialgebras, Sci. China Math. (2018). https://doi.org/10.1007/s11425-016-9102-y, arXiv:1605.00722.

\bibitem{CP}
V. Chari and A. Pressley, A Guide to Quantum Groups, Cambridge University Press, Cambridge (1994).

\bibitem{D} V. Drinfeld, Hamiltonian structure on the Lie groups, Lie bialgebras and the geometric sense of the classical Yang-Baxter equations, {\em Soviet Math. Dokl.} \textbf{27} (1983), 68.

\bibitem{HLS}
J.~T. Hartwig, D.~Larsson, and S.~D. Silvestrov, Deformations of Lie
algebras using $\sigma$-derivations, J. Algebra \textbf{295} (2006) 314.

\bibitem{LS1}
D.~Larsson and S.~D. Silvestrov, Quasi-hom-Lie algebras, central
extensions and 2-cocycle-like identities, J. Algebra \textbf{288} (2005) 321.

\bibitem{LS2}
D.~Larsson and S.~D. Silvestrov, Quasi-Lie algebras, Contemp. Math.
  \textbf{391} (2005) 241.

\bibitem{MDL}  T. Ma, L. Dong and H. Li, General Hom-Lie algebras, J. Algebra Appl. {\bf 15} (2016), 1650081.

\bibitem{MLY}  T. Ma, H. Li, T. Yang, Cobraided smash product Hom-Hopf algebras, Colloq. Math. {\bf 134} (2014), 75-92.

\bibitem{MS2}
A.~Makhlouf and S.~D. Silvestrov, Hom-algebra Structures, J. Gen.
Lie Theory Appl. \textbf{2} (2008) 51.

\bibitem{MS1}
A.~Makhlouf and S.~Silvestrov, Notes on 1-parameter formal
deformations of
  Hom-associative and Hom-Lie algebras, Forum Math. \textbf{22} (2010) 715.

\bibitem{S}
Y.~Sheng, Representations of hom-Lie algebras, Algebr. Represent.
Theory, \textbf{15} (2012) 1081.

\bibitem{BS}
Y.~Sheng and C.~Bai, A new approach to hom-Lie bialgebras, J.
Algebra
  \textbf{399} (2014) 232.

\bibitem{SD}
Y.~Sheng and D.~Chen, Hom-Lie 2-algebras, J. Algebra \textbf{376}
(2013) 174.

\bibitem{Y1}
D.~Yau, Hom-algebras and homology, J. Lie Theory \textbf{19}
(2009) 409.

\bibitem{DY}
D.~Yau, Hom-Novikov algebras, J. Phys. A \textbf{44} (2011) 085202.

\bibitem{DY-hlbi}
D.~Yau, The classical Hom-Yang-Baxter equation and Hom-Lie
bialgebras, Int.
  Electron. J. Algebra \textbf{17} (2015) 11.

\end{thebibliography}
\end{document}